\documentclass[12pt]{amsart}
\usepackage{txfonts}      
\usepackage{amssymb}
\usepackage{eucal}
\usepackage{amsmath}
\usepackage{amscd}
\usepackage{xcolor}
\usepackage{multicol}
\usepackage[all]{xy}           
\usepackage{graphicx}
\usepackage{color}
\usepackage{colordvi}
\usepackage{xspace}
\usepackage{tikz}
\usepackage{makecell}
\usepackage{appendix}
\usepackage{amsthm}

\usepackage{ifpdf}
\ifpdf
\usepackage[colorlinks,final,backref=page,hyperindex]{hyperref}
\else
\usepackage[colorlinks,final,backref=page,hyperindex,hypertex]{hyperref}
\fi

\usepackage[active]{srcltx} 

\begin{document}


\newtheorem{thm}{Theorem}[section]
\newtheorem{lem}[thm]{Lemma}
\newtheorem{cor}[thm]{Corollary}
\newtheorem{pro}[thm]{Proposition}
\theoremstyle{definition}
\newtheorem{defi}[thm]{Definition}
\newtheorem{ex}[thm]{Example}
\newtheorem{cla}[thm]{Claim}
\newtheorem{rmk}[thm]{Remark}
\newtheorem{pdef}[thm]{Proposition-Definition}
\newtheorem{condition}[thm]{Condition}

\renewcommand{\labelenumi}{{\rm(\alph{enumi})}}
\renewcommand{\theenumi}{\alph{enumi}}

\newcommand {\emptycomment}[1]{} 

\newcommand{\nc}{\newcommand}
\newcommand{\delete}[1]{}

\nc{\tred}[1]{\textcolor{red}{#1}}
\nc{\tblue}[1]{\textcolor{blue}{#1}}
\nc{\tgreen}[1]{\textcolor{green}{#1}}
\nc{\tpurple}[1]{\textcolor{purple}{#1}}
\nc{\tgray}[1]{\textcolor{gray}{#1}}
\nc{\torg}[1]{\textcolor{orange}{#1}}
\nc{\tmag}[1]{\textcolor{magenta}}
\nc{\btred}[1]{\textcolor{red}{\bf #1}}
\nc{\btblue}[1]{\textcolor{blue}{\bf #1}}
\nc{\btgreen}[1]{\textcolor{green}{\bf #1}}
\nc{\btpurple}[1]{\textcolor{purple}{\bf #1}}

\nc{\todo}[1]{\tred{To do:} #1}

    \nc{\mlabel}[1]{\label{#1}}  
    \nc{\mcite}[1]{\cite{#1}}  
    \nc{\mref}[1]{\ref{#1}}  
    \nc{\meqref}[1]{\eqref{#1}}  
    \nc{\mbibitem}[1]{\bibitem{#1}} 

\delete{
    \nc{\mlabel}[1]{\label{#1}    { {\small\tgreen{\tt{{\ }(#1)}}}}} 
    \nc{\mcite}[1]{\cite{#1}{\small{\tt{{\ }(#1)}}}}  
    \nc{\mref}[1]{\ref{#1}{\small{\tred{\tt{{\ }(#1)}}}}}  
    \nc{\meqref}[1]{\eqref{#1}{{\tt{{\ }(#1)}}}}  
    \nc{\mbibitem}[1]{\bibitem[\bf #1]{#1}} 
}


\nc{\hn}[1]{\textcolor{red}{Henan:#1}}
\nc{\yy}[1]{\textcolor{blue}{Yanyong: #1}}
\nc{\li}[1]{\textcolor{purple}{#1}}
\nc{\lir}[1]{\textcolor{purple}{Li:#1}}

\nc{\revise}[1]{\textcolor{blue}{#1}}


\nc{\tforall}{\ \ \text{for all }}
\nc{\hatot}{\,\widehat{\otimes} \,}
\nc{\complete}{completed\xspace}
\nc{\wdhat}[1]{\widehat{#1}}

\nc{\ts}{\mathfrak{p}}
\nc{\mts}{c_{(i)}\ot d_{(j)}}

\nc{\NA}{{\bf NA}}
\nc{\LA}{{\bf Lie}}
\nc{\CLA}{{\bf CLA}}
\nc{\gda}{GD algebra\xspace}
\nc{\gdas}{GD algebras\xspace}
\nc{\gdba}{GD bialgebra\xspace}
\nc{\gdbas}{GD bialgebras\xspace}
\nc{\cybe}{CYBE\xspace}
\nc{\nybe}{NYBE\xspace}
\nc{\ccybe}{CCYBE\xspace}

\nc{\ndend}{pre-Novikov\xspace}
\nc{\calb}{\mathcal{B}}
\nc{\rk}{\mathrm{r}}
\newcommand{\pf}{\noindent{$Proof$.}\ }
\newcommand{\frkg}{\mathfrak g}
\newcommand{\frkh}{\mathfrak h}
\newcommand{\Z}{\mathbb{Z}}
\newcommand{\C}{\mathbb{C}}
\newcommand{\ad}{\mathrm{ad}}
\newcommand{\add}{\frka\frkd}
\newcommand{\frka}{\mathfrak a}
\newcommand{\frkb}{\mathfrak b}
\newcommand{\frkc}{\mathfrak c}
\newcommand{\frkd}{\mathfrak d}
\newcommand {\comment}[1]{{\marginpar{*}\scriptsize\textbf{Comments:} #1}}

\nc{\vspa}{\vspace{-.1cm}}
\nc{\vspb}{\vspace{-.2cm}}
\nc{\vspc}{\vspace{-.3cm}}
\nc{\vspd}{\vspace{-.4cm}}
\nc{\vspe}{\vspace{-.5cm}}


\nc{\disp}[1]{\displaystyle{#1}}
\nc{\bin}[2]{ (_{\stackrel{\scs{#1}}{\scs{#2}}})}  
\nc{\binc}[2]{ \left (\!\! \begin{array}{c} \scs{#1}\\
    \scs{#2} \end{array}\!\! \right )}  
\nc{\bincc}[2]{  \left ( {\scs{#1} \atop
    \vspace{-.5cm}\scs{#2}} \right )}  
\nc{\ot}{\otimes}
\nc{\sot}{{\scriptstyle{\ot}}}
\nc{\otm}{\overline{\ot}}
\nc{\ola}[1]{\stackrel{#1}{\la}}

\nc{\scs}[1]{\scriptstyle{#1}} \nc{\mrm}[1]{{\rm #1}}

\nc{\dirlim}{\displaystyle{\lim_{\longrightarrow}}\,}
\nc{\invlim}{\displaystyle{\lim_{\longleftarrow}}\,}

\nc{\bfk}{{\bf k}} \nc{\bfone}{{\bf 1}}
\nc{\rpr}{\circ}
\nc{\dpr}{{\tiny\diamond}}
\nc{\rprpm}{{\rpr}}

\nc{\mmbox}[1]{\mbox{\ #1\ }} \nc{\ann}{\mrm{ann}}
\nc{\Aut}{\mrm{Aut}} \nc{\can}{\mrm{can}}
\nc{\twoalg}{{two-sided algebra}\xspace}
\nc{\colim}{\mrm{colim}}
\nc{\Cont}{\mrm{Cont}} \nc{\rchar}{\mrm{char}}
\nc{\cok}{\mrm{coker}} \nc{\dtf}{{R-{\rm tf}}} \nc{\dtor}{{R-{\rm
tor}}}
\renewcommand{\det}{\mrm{det}}
\nc{\depth}{{\mrm d}}
\nc{\End}{\mrm{End}} \nc{\Ext}{\mrm{Ext}}
\nc{\Fil}{\mrm{Fil}} \nc{\Frob}{\mrm{Frob}} \nc{\Gal}{\mrm{Gal}}
\nc{\GL}{\mrm{GL}} \nc{\Hom}{\mrm{Hom}} \nc{\hsr}{\mrm{H}}
\nc{\hpol}{\mrm{HP}}  \nc{\id}{\mrm{id}} \nc{\im}{\mrm{im}}

\nc{\incl}{\mrm{incl}} \nc{\length}{\mrm{length}}
\nc{\LR}{\mrm{LR}} \nc{\mchar}{\rm char} \nc{\NC}{\mrm{NC}}
\nc{\mpart}{\mrm{part}} \nc{\pl}{\mrm{PL}}
\nc{\ql}{{\QQ_\ell}} \nc{\qp}{{\QQ_p}}
\nc{\rank}{\mrm{rank}} \nc{\rba}{\rm{RBA }} \nc{\rbas}{\rm{RBAs }}
\nc{\rbpl}{\mrm{RBPL}}
\nc{\rbw}{\rm{RBW }} \nc{\rbws}{\rm{RBWs }} \nc{\rcot}{\mrm{cot}}
\nc{\rest}{\rm{controlled}\xspace}
\nc{\rdef}{\mrm{def}} \nc{\rdiv}{{\rm div}} \nc{\rtf}{{\rm tf}}
\nc{\rtor}{{\rm tor}} \nc{\res}{\mrm{res}} \nc{\SL}{\mrm{SL}}
\nc{\Spec}{\mrm{Spec}} \nc{\tor}{\mrm{tor}} \nc{\Tr}{\mrm{Tr}}
\nc{\mtr}{\mrm{sk}}

\nc{\ab}{\mathbf{Ab}} \nc{\Alg}{\mathbf{Alg}}

\nc{\BA}{{\mathbb A}} \nc{\CC}{{\mathbb C}} \nc{\DD}{{\mathbb D}}
\nc{\EE}{{\mathbb E}} \nc{\FF}{{\mathbb F}} \nc{\GG}{{\mathbb G}}
\nc{\HH}{{\mathbb H}} \nc{\LL}{{\mathbb L}} \nc{\NN}{{\mathbb N}}
\nc{\QQ}{{\mathbb Q}} \nc{\RR}{{\mathbb R}} \nc{\BS}{{\mathbb{S}}} \nc{\TT}{{\mathbb T}}
\nc{\VV}{{\mathbb V}} \nc{\ZZ}{{\mathbb Z}}


\nc{\calao}{{\mathcal A}} \nc{\cala}{{\mathcal A}}
\nc{\calc}{{\mathcal C}} \nc{\cald}{{\mathcal D}}
\nc{\cale}{{\mathcal E}} \nc{\calf}{{\mathcal F}}
\nc{\calfr}{{{\mathcal F}^{\,r}}} \nc{\calfo}{{\mathcal F}^0}
\nc{\calfro}{{\mathcal F}^{\,r,0}} \nc{\oF}{\overline{F}}
\nc{\calg}{{\mathcal G}} \nc{\calh}{{\mathcal H}}
\nc{\cali}{{\mathcal I}} \nc{\calj}{{\mathcal J}}
\nc{\call}{{\mathcal L}} \nc{\calm}{{\mathcal M}}
\nc{\caln}{{\mathcal N}} \nc{\calo}{{\mathcal O}}
\nc{\calp}{{\mathcal P}} \nc{\calq}{{\mathcal Q}} \nc{\calr}{{\mathcal R}}
\nc{\calt}{{\mathcal T}} \nc{\caltr}{{\mathcal T}^{\,r}}
\nc{\calu}{{\mathcal U}} \nc{\calv}{{\mathcal V}}
\nc{\calw}{{\mathcal W}} \nc{\calx}{{\mathcal X}}
\nc{\CA}{\mathcal{A}}

\nc{\fraka}{{\mathfrak a}} \nc{\frakB}{{\mathfrak B}}
\nc{\frakb}{{\mathfrak b}} \nc{\frakd}{{\mathfrak d}}
\nc{\oD}{\overline{D}}
\nc{\frakF}{{\mathfrak F}} \nc{\frakg}{{\mathfrak g}}
\nc{\frakm}{{\mathfrak m}} \nc{\frakM}{{\mathfrak M}}
\nc{\frakMo}{{\mathfrak M}^0} \nc{\frakp}{{\mathfrak p}}
\nc{\frakS}{{\mathfrak S}} \nc{\frakSo}{{\mathfrak S}^0}
\nc{\fraks}{{\mathfrak s}} \nc{\os}{\overline{\fraks}}
\nc{\frakT}{{\mathfrak T}}
\nc{\oT}{\overline{T}}
\nc{\frakX}{{\mathfrak X}} \nc{\frakXo}{{\mathfrak X}^0}
\nc{\frakx}{{\mathbf x}}
\nc{\frakTx}{\frakT}      
\nc{\frakTa}{\frakT^a}        
\nc{\frakTxo}{\frakTx^0}   
\nc{\caltao}{\calt^{a,0}}   
\nc{\ox}{\overline{\frakx}} \nc{\fraky}{{\mathfrak y}}
\nc{\frakz}{{\mathfrak z}} \nc{\oX}{\overline{X}}


\title{Generalized conformal modules over the Virasoro conformal algebra}

\author{Henan Wu}
\address{School of Mathematics and Statistics, Shanxi University, Taiyuan 030006, PR China}
\email{wuhenan@sxu.edu.cn}

\author{Yanyong Hong}
\address{School of Mathematics, Hangzhou Normal University,
Hangzhou 311121, PR China}
\email{yyhong@hznu.edu.cn}
\thanks{Corresponding author: Yanyong Hong. Email: yyhong@hznu.edu.cn}
\subjclass[2010]{
17B10, 17B68, 17B69, 81R10, 81T40
}

\keywords{Lie conformal algebras, Virasoro conformal algebra, generalized conformal modules, annihilation Lie algebras, irreducible modules}

\begin{abstract}
\delete{In this paper, we initiate a systematic study of generalized conformal modules over the Virasoro conformal algebra $Vir$, a notion originally defined by V.~Kac but largely unexplored in the literature.
First, we present a twisted construction that produces new generalized conformal modules over any Lie conformal algebra from existing conformal modules.
Then we give a complete classification of non-trivial generalized conformal modules over $Vir$ that are free of rank one over $\mathbb C[\partial]$, and show that they are precisely those obtained from rank-one conformal modules via the twisted construction.
As an application, we classify all modules over the Lie algebra \(W_1\) of vector fields on a line that are free of rank one over $\mathbb C[L_{-1}]$, with explicit actions expressed in terms of Bell polynomials.
Furthermore, we develop a matrix exponential construction that yields irreducible generalized conformal modules over $Vir$ of arbitrary finite rank, revealing a striking phenomenon with no analogue in the conformal setting.
We also construct and classify a family of infinite torsion generalized conformal modules over $Vir$, which are non-trivial and irreducible under suitable conditions.}
\delete{Let $W_1=\bigoplus_{i\geq -1}\mathbb{C}L_i$ be the Lie algebra of vector fields on a line. In this paper, we study a class of non-weight modules over $W_1$,  which are free of rank one over $\mathbb{C}[L_{-1}]$. By establishing a correspondence   between such $W_1$-modules and (may be \hn{not} conformal) modules of the Virasoro conformal algebra $Vir$ that are free of rank \hn{one} over $\mathbb{C}[\partial]$, we first characterize the corresponding modules  over $Vir$ and then provide a complete classification of  $W_1$-modules under consideration by using Bell polynomials. Moreover, we investigate the simplicities and isomorphism classes of these modules.}
This paper investigates generalized conformal modules over the Virasoro conformal algebra $Vir$, a notion originally defined by V.~Kac but largely unexplored in the literature. First, we present a twisted construction that produces new generalized conformal modules over any Lie conformal algebra from existing conformal modules. Then we give a complete classification of non-trivial generalized conformal modules over $Vir$ that are free of rank one over $\mathbb C[\partial]$, which in turn yields a classification of  modules over the Lie algebra \(W_1=\bigoplus_{i\geq -1}\C L_i\) of vector fields on a line that are free of rank one over $\mathbb C[L_{-1}]$ with explicit actions expressed in terms of Bell polynomials. For arbitrary finite rank $n$, we construct a class of generalized conformal modules $V_{a,b,C(\partial)}$ and establish a complete irreducibility criterion in terms of the differential operator $\mathcal L=\frac{d}{d\partial}+C(\partial)$ acting on $\mathbb{C}(\partial)^n$. In the case of $n=2$, this criterion reduces to the absence of rational solutions of a Riccati equation. Then we prove the existence of irreducible generalized conformal modules of arbitrary finite rank, revealing a striking contrast with the conformal setting. We also construct and classify a family of infinite torsion generalized conformal modules over $Vir$, which are non-trivial and irreducible under suitable conditions.
\end{abstract}

\maketitle




\allowdisplaybreaks
\section{Introduction}
Introduced by V. Kac in \cite{K1}, Lie conformal algebras provide an axiomatic framework for the singular part of operator product expansions in conformal field theory \cite{BPZ}. Also known as  vertex Lie algebras \cite{P}, they are closely related to vertex algebras \cite{K1}, infinite-dimensional Lie algebras satisfying the
 locality property \cite{KacL} and Hamiltonian formalism in the theory of nonlinear evolution equations
 \cite{BDK}.  A Lie conformal algebra is called finite if it is finitely generated as a $\C[\partial]$-module. Structure theory and cohomology theory of finite Lie conformal algebras have been well developed (see \cite{DK, BKV, DK1}).

 Among all Lie conformal algebras, the Virasoro conformal algebra $Vir$
 is  undoubtedly the most fundamental and widely studied example. It is a finite simple Lie conformal algebra \cite{DK} and occupies a central position in conformal field theory.  Its algebraic significance is further highlighted by the fact that its coefficient algebra is isomorphic to the Witt algebra (the derivation Lie algebra of the Laurent polynomial ring $\C[x,x^{-1}]$), while its annihilation Lie algebra is isomorphic to the Lie algebra $W_1=\bigoplus_{i\geq -1}\C L_i$ of vector fields on a line which is the derivation Lie algebra of the polynomial ring $\C[x]$. Moreover, $Vir$ appears as a subalgebra of many important Lie conformal algebras such as the Schr\"odinger-Virasoro Lie conformal algebra \cite{SY}, Lie conformal algebras of Block type \cite{SXY}, and the general Lie conformal algebra $gc_N$ \cite{DK2}.

 A {\bf conformal module} over a Lie conformal algebra $\mathcal {A}$ is a $\C[\partial]$-module $V$ together with a $\C$-bilinear map $\mathcal {A}\times V\rightarrow V[\lambda]$, $(a,v)\mapsto a_\lambda v$, satisfying the standard axioms \cite{CK}. Note that by \cite{K1},  conformal modules over $\mathcal {A}$ are equivalent to restricted modules over the extended annihilation Lie algebra of $\mathcal {A}$ and there is an equivalence between the category of conformal modules over $\mathcal {A}$ and a certain category of modules over its coefficient algebra. Consequently, the study of conformal modules provides a powerful tool for understanding representations of infinite-dimensional Lie algebras. Indeed, conformal modules over various Lie conformal algebras have been extensively investigated. For instance,
finite irreducible conformal modules over finite simple Lie conformal algebras including $Vir$ have been completely classified via the representation theory of their extended annihilation Lie algebras \cite{CK}. It was shown that finite non-trivial irreducible conformal modules over $Vir$ are of rank one.

However, the class of conformal modules only captures those representations of the extended annihilation Lie algebra that are restricted in a strong sense, namely, for any $v\in V$ and $a\in \mathcal {A}$, there exists $N$ such that $a_{(n)}v=0$ for all $n>N$. This restriction, while natural from the viewpoint of vertex algebras and conformal field theory, leaves out a vast family of modules over infinite-dimensional Lie algebras that are of independent interest. In order to  study a broader  class of representations including those that are not necessarily restricted, it is desirable to generalize the notion of a conformal module by relaxing the polynomial condition $a_\lambda v\in V[\lambda]$ to allow $a_\lambda v$ to be a formal power series in $\lambda$. This gives the notion of {\bf generalized conformal modules}, which were already defined by V. Kac \cite{K1} but have not been systematically studied. The present paper initiates such a study, focusing on investigating generalized conformal modules over the Virasoro conformal algebra $Vir$.

In this paper, we first present a twisted construction that produces new generalized conformal modules over any Lie conformal algebra $\mathcal {A}$ from existing conformal modules over $\mathcal {A}$.  This construction demonstrates the richness of the theory of generalized conformal modules, as it yields modules that are not conformal in the usual sense.
It is well known that for any conformal module $V$ over a Lie conformal algebra $\mathcal {A}$,  $\mathcal {A}$ acts trivially on the torsion of $V$, i.e., $\mathcal{A}_\lambda \text{Tor}~(V)=0$. It follows immediately that the irreducible torsion conformal module is a one-dimensional trivial module.
However, this result does not hold for the generalized conformal modules. In fact, it only holds for the
finite generalized conformal modules.
We show that there exist infinite non-trivial irreducible torsion generalized conformal modules over $Vir$. Moreover, we give a complete classification of a class of infinite torsion generalized conformal modules over $Vir$ and investigate their irreducibilities and isomorphism classes. This reveals a fundamental difference between conformal modules and generalized conformal modules. \delete{The key distinctions are summarised in the following table.
\[
\begin{array}{c|c|c}
\text{irreducible torsion modules} & \text{conformal} & \text{generalized conformal} \\
\hline
\text{finite} & \text{trivial} & \text{trivial} \\
\text{infinite} & \text{trivial} & \text{may not trivial}
\end{array}
\]}

Through a series of skillful computations, we give a complete classification of non-trivial generalized conformal modules over $Vir$, which are free of rank one over $\C[\partial]$. These generalized conformal modules are precisely those arising from rank-one conformal modules over $Vir$ via the twisted construction mentioned above. The irreducibilities and isomorphism classes of these generalized conformal modules are also investigated. As a result,  we classify
 all $W_1$-modules that are free of rank one over $\C[L_{-1}]$, with explicit actions expressed in terms of Bell polynomials. This recovers a previous classification given in \cite{HCL}, where such modules appeared in a different form.

 \delete{Note that by \cite{CK}, finite non-trivial irreducible  conformal modules over $Vir$ are of rank one. It is therefore natural to ask whether there exist non-trivial irreducible generalized conformal modules of higher ranks over $Vir$. To address this question, we first develop a general construction of generalized conformal modules of arbitrary rank $n$ over $Vir$, starting from a given matrix $C(\partial)\in (\C[\partial])^{n\times n}$. In particular, by choosing a suitable $C(\partial)$ and using the representation theory of $W_1$, we obtain irreducible generalized conformal modules of arbitrary finite rank over $Vir$, which is a new phenomenon that has no analogue in the conformal setting. The stark contrast is summarised in the following table.}
 It is known from \cite{CK} that every finite non-trivial irreducible conformal module over $Vir$ is of rank one. A natural question is whether this rigidity persists for generalized conformal modules. In this paper, we show that the answer is negative. We construct a class of generalized conformal modules $V_{a,b,C(\partial)}=\C[\partial]^n$, parametrized by $a$, $b\in \mathbb{C}$ and a matrix $C(\partial)\in \mathbb C[\partial]^{n\times n}$, with the action given by $L_\lambda = (\partial+a\lambda+b)e^{\lambda \mathcal L}$ where $\mathcal L=\frac{d}{d\partial}+C(\partial)$. We prove that $V_{a,b,C(\partial)}$ is irreducible if and only if $a\neq 0$ and $\mathcal L$ acts irreducibly on $\mathbb C(\partial)^n$. Equivalently, for any $1\leq r<n$ and $B(\partial)\in \C(\partial)^{r\times r}$, the matrix differential equation $Y'(\partial)+C(\partial)Y(\partial)=Y(\partial)B(\partial)$ has no nonzero solutions $Y(\partial)\in \C(\partial)^{n\times r}$. For $n=2$, this criterion reduces to checking the absence of rational solutions of a Riccati equation. We classify their isomorphism classes: $V_{a_1,b_1,C_1(\partial)}\cong V_{a_2,b_2,C_2(\partial)}$ if and only if $(a_1,b_1)=(a_2,b_2)$ and there exists an invertible matrix $G(\partial)\in  \mathbb{C}[\partial]^{n\times n}$ such that $G'(\partial)=G(\partial)C_1(\partial)-C_2(\partial)G(\partial)$. We also give a criterion for when such modules are conformal. By choosing $C(\partial)=\begin{pmatrix}0&\partial\\ I_{n-1}&0\end{pmatrix}$, we prove the existence of irreducible generalized conformal modules of arbitrary finite rank,  which is a new phenomenon that has no analogue in the conformal setting. The stark contrast is summarized in the following table.
 \[
\begin{array}{c|c}
\text{modules} & \text{finite non-trivial irreducible} \\
\hline
\text{conformal modules} & \text{of rank one} \\
\text{generalized conformal modules} & \text{can be of arbitrary rank}
\end{array}
\]

Overall, this paper investigates generalized conformal modules over Lie conformal algebras, a direction that was defined by V. Kac but remained largely unexplored. Our results reveal that the generalized setting admits a representation theory vastly richer than the classical conformal case, including torsion modules of infinite rank and free modules of arbitrary finite rank. Moreover, by establishing explicit connections with modules over $W_1$, we provide a unified framework that not only recovers known results but also yields new classes of representations of infinite-dimensional Lie algebras. We believe that this work opens new perspectives for the representation theory of Lie conformal algebras and their associated infinite-dimensional Lie algebras.

From the perspective of Lie conformal algebras, the Virasoro conformal algebra $Vir$ is one of the most fundamental examples and serves as a testing ground for the general theory. Its representation theory is not only of intrinsic algebraic interest but also has close ties to conformal field theory via the correspondence between conformal modules and restricted modules over the associated annihilation Lie algebra. The generalized conformal modules introduced in this paper extend this correspondence beyond the restricted setting, thereby providing a broader algebraic framework that may be relevant to non-diagonalisable structures arising in logarithmic conformal field theory and related models. In particular, the explicit classification of generalized conformal modules over $Vir$ obtained here offers a systematic framework for understanding representations beyond the classical conformal setting, and we expect it to contribute to a more complete algebraic picture of such structures.

This paper is organized as follows. In Section~2, we recall the necessary background on Lie conformal algebras and their (extended) annihilation Lie algebras, as well as the correspondence between generalized conformal modules over a Lie conformal algebra and modules over its extended annihilation Lie algebra. We also present a twisted construction that produces generalized conformal modules from existing conformal modules, and establish a structural result showing that any finite irreducible generalized conformal module over a Lie conformal algebra is either trivial or free of finite rank over $\C[\partial]$. Section~3 is devoted to the classification of non-trivial generalized conformal modules over $Vir$ that are free of rank one over $\C[\partial]$. The irreducibility criterion and isomorphism classes of these modules are also provided.
Moreover, as an application, we give a classification of modules over $W_1$ that are free of rank one over $\C[L_{-1}]$.
In Section 4, we construct a class of generalized conformal modules $V_{a,b,C(\partial)}$ of arbitrary finite rank $n$ and establish a complete irreducibility criterion in terms of the differential operator $\mathcal L=\frac{d}{d\partial}+C(\partial)$ acting on $\mathbb C(\partial)^n$. In the case of $n=2$, the criterion reduces to the absence of rational solutions of a Riccati equation. We also classify isomorphism classes, characterize when these modules are conformal, and prove the existence of irreducible generalized conformal modules of arbitrary finite rank, revealing a striking contrast with the conformal setting. Section~5 constructs and classifies a family of infinite torsion generalized conformal modules over $Vir$ and determines their irreducibilities and isomorphism classes.

Throughout this paper, we denote by $\C$, $\Z$,  $\Z_+$ and $\Z_{\geq -1}$ the sets of complex numbers, integer numbers, nonnegative integers and integers that are greater than or equal to  1. Moreover, if $A$ is a vector space over $\C$, then we denote the vector spaces of polynomials and power series in $\lambda$ with coefficients in $A$  by $A[\lambda]$ and $A[[\lambda]]$ respectively. For a nonzero polynomial $f(\partial,\lambda)\in \C[\partial, \lambda]$, we denote the degree of $f(\partial,\lambda)$ in $\lambda$ by $\text{deg}_\lambda(f(\partial,\lambda))$.

\section{Preliminaries on generalized conformal modules over Lie conformal algebras}
In this section, we recall the basic definitions and properties of Lie conformal algebras, their annihilation Lie algebras, and their extended annihilation Lie algebras. We then introduce the notion of generalized conformal modules  and present a twisted construction that produces new generalized conformal modules from old ones. Finally, we analyze the structure of finite irreducible generalized conformal modules over an arbitrary Lie conformal algebra, showing that they are either trivial or free of finite rank as $\C[\partial]$-modules. These results lay the foundation for the classification of finite irreducible generalized conformal modules over $Vir$.
\subsection{Lie conformal algebras and extended annihilation Lie algebras}
First, we recall the definition of Lie conformal algebras.
\begin{defi}
\cite{DK}
A {\bf Lie conformal algebra} $\mathcal {A}$ is a $\C[\partial ]$-module endowed with a $\C$-bilinear map $\mathcal {A}\times \mathcal {A}\rightarrow \mathcal{A}[\lambda]$, $a\times b \mapsto [a_\lambda b]
$ subject to the following relations for any $a, b, c\in \mathcal {A}$:
\begin{align*}
[\partial a_\lambda b]&=-\lambda[a_\lambda b],\ \ [ a_\lambda \partial b]=(\partial+\lambda)[a_\lambda b] \ \ \mbox{(conformal\  sesquilinearity)},
\\
{[a_\lambda b]} &= -[b_{-\lambda-\partial}a] \ \ \mbox{(skew-symmetry)},
\\
{[a_\lambda[b_\mu c]]}&=[[a_\lambda b]_{\lambda+\mu
}c]+[b_\mu[a_\lambda c]]\ \ \mbox{(Jacobi \ identity)}.
\end{align*}
A Lie conformal algebra is called {\bf finite} if it is finitely generated as a $\C[\partial]$-module.
\end{defi}

\begin{ex}
The {\bf Virasoro conformal algebra} $Vir=\mathbb{C}[\partial]L$ is a free $\mathbb{C}[\partial]$-module of rank one, whose $\lambda$-brackets are determined by
\begin{eqnarray*}
[L_\lambda L]=(\partial+2\lambda)L.
\end{eqnarray*} Furthermore, $Vir$ is a simple Lie conformal algebra (see \cite{DK}).
\end{ex}

Next, let us recall the notion of the (extended) annihilation Lie algebra of a Lie conformal algebra.
\begin{defi} \cite{K1} The {\bf annihilation Lie algebra} $\text{Lie}( \mathcal{A})^+$ of a Lie conformal algebra $ \mathcal{A}$ is the vector space $\text{span}_\mathbb{C}\{a_{(n)}\mid a \in \mathcal{A} , \ n\in \mathbb{Z}_+\}$ with relations
\begin{eqnarray*}(\partial a)_{(n)}=-na_{(n-1)},\ (a+b)_{(n)}=a_{(n)}+b_{(n)},\ (ka)_{(n)}=ka_{(n)}\; \;\text{for any $a,b \in \mathcal{A}$ and $k \in \mathbb{C}$},
\end{eqnarray*}
and the Lie brackets of $\text{Lie}( \mathcal{A})^+$ are given by
\begin{eqnarray*}
[a_{(m)},b_{(n)}]= \sum\limits_{i\in \mathbb{Z}_+} \binom m i (a_{(i)}b)_{(m+n-i)}\;\;\text{for any $a$, $b\in \mathcal {A}$, $m$, $n\in \Z_+$.}
\end{eqnarray*}
The {\bf extended annihilation Lie algebra} $\text{Lie}( \mathcal{A})^e$ of $\mathcal{A}$ is the semi-direct product of Lie algebras $\mathbb{C}\partial \ltimes \text{Lie}( \mathcal{A})^+$ with the Lie brackets
\begin{eqnarray*}
[\partial, a_{(n)}]=-na_{(n-1)}\;\;\text{for any $a\in \mathcal{A}$, $n\in \Z_+$,}
\end{eqnarray*} which contains $\text{Lie}( \mathcal{A})^+$  as a subalgebra.
\end{defi}

\begin{ex}
The annihilation Lie algebra $\text{Lie}({Vir})^+$  of $Vir$ is a Lie algebra with a basis $\{L_{(i)}\}_{i\in \Z_+}$ such that
\begin{eqnarray}
[L_{(m)},L_{(n)}]=(m-n)L_{(m+n-1)} \;\;\text{for any $m$, $n\in \Z_+$.}
\end{eqnarray}
The extended annihilation Lie algebra $\text{Lie}( Vir)^e$ of $Vir$ is a Lie algebra with a basis $\{\partial, \{L_{(i)}\}_{i\in \Z_+}\}$ with the Lie brackets given by
\begin{eqnarray*}
[\partial, L_{(n)}]=-nL_{(n-1)},\;\;[L_{(m)},L_{(n)}]=(m-n)L_{(m+n-1)} \;\;\text{for any $m$, $n\in \Z_+$.}
\end{eqnarray*}
Note that $\partial-L_{(0)}$ is in the center of $\text{Lie}( Vir)^e$. Consequently, we have
$$\text{Lie}( Vir)^e=\text{Lie}( Vir)^+\oplus \C(\partial-L_{(0)})$$
as a direct sum of Lie algebras.
\end{ex}

Recall that the Lie algebra $W_1=\bigoplus_{i\geq -1}\mathbb{C}L_i$ of vector fields on a line is defined by the commutation relation
\begin{eqnarray*}
[L_i, L_j]=(i-j)L_{i+j}\;\;\text{for any $i$, $j\in \Z_{\geq -1}$.}
\end{eqnarray*}

\begin{pro}\label{pro-1}
$\text{Lie}( Vir)^+$ is isomorphic to $W_1$ as Lie algebras.
\end{pro}
\begin{proof}
Let $\varphi: \text{Lie}( Vir)^+\rightarrow W_1 $ be the linear map defined by $\varphi(L_{(i)})=L_{i-1}$ for each $i\in \Z_{+}$. Then it is easy to see that $\varphi$ is an isomorphism of Lie algebras.
\end{proof}

By Proposition \ref{pro-1}, we have $\text{Lie}( Vir)^e\cong \C \partial \ltimes W_1$ with
the brackets $[\partial, L_n]=(-1-n)L_{n-1}$ for any $n\in \Z_{\geq -1}$. Consequently, we have
$$\text{Lie}( Vir)^e\cong W_1\oplus \C(\partial-L_{-1})$$
as a direct sum of Lie algebras. Hence a module $V$ over $\text{Lie}( Vir)^e$ is naturally 
a module over $W_1$ and each module over $W_1$ can be lifted to be a module over $\text{Lie}( Vir)^e$. In fact, we can also have the following result.
\begin{cor}\label{cor-4}
A module $V$ over $W_1$ can be seen as a module over $\text{Lie}( Vir)^e$ with
\begin{eqnarray}
L_{(i)}. v=L_{i-1}. v,\;\; \partial. v=L_{(0)}.v+\xi v=L_{-1}.v+\xi v\;\;\text{for any $i\in \Z_+$ and $v\in V$,}
\end{eqnarray}
for some $\xi\in \C$. Denote this module over $\text{Lie}( Vir)^e$ by $V_\xi$. 
Moreover, we have the following conclusions:
\begin{enumerate}
\item The map $V\mapsto V_\xi$ gives a one-to-one correspondence between $W_1$-modules and $\text{Lie}( Vir)^e$-modules on which $\partial-L_{(0)}$ acts as the scalar $\xi$.
\item $V$ is finitely generated over $\C[L_{-1}]$ if and only if $V_\xi$ is finitely generated over $\C[\partial]$.
\item $V$ is free of rank $n$ over $\C[L_{-1}]$ if and only if $V_\xi$ is free of rank $n$ over $\C[\partial]$.
\item $V$ is an irreducible module over $W_1$ if and only if $V_\xi$ is an irreducible  module over $\text{Lie}( Vir)^e$.
\item Any irreducible module over $\text{Lie}(Vir)^e$ is isomorphic to some $V_{\xi}$, where $V$ is an irreducible module over $W_1$ and $\xi\in \C$.
\end{enumerate}
\end{cor}
\begin{proof}
It is straightforward.
\end{proof}
\subsection{Generalized conformal modules and a twisted construction}
\begin{defi}\cite{K1}
A {\bf generalized conformal module} $V$ over a Lie conformal algebra $\mathcal {A}$
is a $\mathbb{C}[\partial]$-module equipped with a $\C$-bilinear map
$\mathcal {A}\times V\rightarrow
V[[\lambda]]$, $a\times v\mapsto a_\lambda v$, satisfying the following relations for any $a,b\in\mathcal {A}$, $v\in V$,
\begin{eqnarray}
&&\label{conf-1}(\partial a)_\lambda v=-\lambda a_\lambda v,\ a_\lambda(\partial
v)=(\partial+\lambda)a_\lambda v,\\
&&\label{conf-2}a_\lambda(b_\mu v)-b_\mu(a_\lambda v)=[a_\lambda b]_{\lambda+\mu}v.
\end{eqnarray}
In particular, a generalized conformal module $V$ over $\mathcal {A}$ is called {\bf conformal}, if for any $a\in \mathcal {A}$ and $v\in V$, we have
$a_\lambda v\in V[\lambda]$.
If $V$ is finitely generated over $\mathbb{C}[\partial]$, then $V$ is simply called {\bf finite}. If $a_\lambda v=0$ for any $a\in \mathcal {A}$ and $v\in V$, then $V$ is called a {\bf trivial generalized conformal module} over $\mathcal {A}$; otherwise, it is called {\bf non-trivial}.
\end{defi}

\begin{rmk}
What we call generalized conformal modules are referred to simply as modules  in \cite{K1}. The term ``generalized" is used here to stress the contrast with conformal modules.

By Eq. (\ref{conf-1}), if $\mathcal {A}=\C[\partial]\otimes A$ and $V=  \C[\partial]\otimes M$ are free as $\C[\partial]$-modules over the vector spaces $A$ and $M$ respectively,  then  the generalized conformal module $V$ over $\mathcal {A}$ is completely determined by the actions of $A$ on $M$.

\delete{It should be also emphasized that a generalized conformal module over a Lie conformal algebra $\mathcal {A}$  may not be conformal. As we know, conformal modules over Lie conformal algebras have been well investigated. In this paper, we mainly consider those generalized conformal modules over Lie conformal algebras which are not conformal.}
\end{rmk}

Note that for an arbitrary polynomial $q(\partial)\in \mathbb{C}[\partial]$,
$$e^{q(\lambda+\partial)-q(\partial)}:=\sum_{k=0}^\infty \frac{(q(\lambda+\partial)-q(\partial))^k}{k!}$$ is well defined in $\mathbb{C}[\partial][[\lambda]]$. Then we present a twisted method to construct generalized conformal modules over Lie conformal algebras.
\begin{pro}\label{Pro-tw1}
Let $\mathcal{A}$ be a Lie conformal algebra, $V$ be a generalized conformal module over $\mathcal{A}$ with the action $a_\lambda v$ for any $a\in \mathcal{A}$, $v\in V$ and $q(\partial)\in \mathbb{C}[\partial]$. 
Define
\begin{eqnarray}\label{new-action}
\overline{a_\lambda v}:=e^{q(\lambda+\partial)-q(\partial)}a_\lambda v\;\;\;\;\;\;\text{for any $a\in \mathcal{A}$, $v\in V$.}
\end{eqnarray}
Then $V$ is a generalized conformal module over $\mathcal{A}$ with the new action. Denote this generalized conformal module by $V_{q(\partial)}$. Moreover, $V_{q(\partial)}$ is irreducible if and only if $V$ is irreducible.
\end{pro}
\begin{proof}
Obviously, the action defined by Eq. (\ref{new-action}) satisfies Eq. (\ref{conf-1}). Let $a$, $b\in\mathcal{A}$ and $v\in V$. Then we obtain
\begin{eqnarray*}
&&\overline{a_\lambda (\overline{b_\mu v})}-\overline{b_\mu (\overline{a_\lambda v})}-\overline{[a_\lambda b]_{\lambda+\mu}v}\\
&&\quad=\overline{a_\lambda (e^{q(\mu+\partial)-q(\partial)}b_\mu v)}-\overline{b_\mu (e^{q(\lambda+\partial)-q(\partial)}a_\lambda v)}-e^{q(\lambda+\mu+\partial)-q(\partial)}[a_\lambda b]_{\lambda+\mu}v\\
&&\quad= e^{q(\lambda+\partial)-q(\partial)}e^{q(\lambda+\mu+\partial)-q(\lambda+\partial)}a_\lambda(b_\mu v)-e^{q(\lambda+\partial+\mu)-q(\mu+\partial)}e^{q(\mu+\partial)-q(\partial)}b_\mu(a_\lambda v)-e^{q(\lambda+\mu+\partial)-q(\partial)}[a_\lambda b]_{\lambda+\mu}v\\
&&\quad=e^{q(\lambda+\mu+\partial)-q(\partial)}(a_\lambda(b_\mu v)-b_\mu(a_\lambda v)-[a_\lambda b]_{\lambda+\mu}v)\\
&&\quad=0.
\end{eqnarray*}
Therefore, $V$ is a generalized conformal module over $\mathcal{A}$ with the new action given by Eq. (\ref{new-action}).

Since $e^{q(\lambda+\partial)-q(\partial)}$ is invertible in $\mathbb{C}[\partial][[\lambda]]$, it is easy to see that $V_{q(\partial)}$ is irreducible if and only if $V$ is irreducible.
\end{proof}
\begin{rmk}\label{Rmk-tw}
Actually, twisting by $q(\partial)$ not only gives an invertible map between generalized conformal modules, but also induces a functor from the category of generalized conformal modules to itself.
Assume that $\phi$ is a module homomorphism from $U$ to $V$. Then it can be checked directly that
$$\phi_{q(\partial)}: U_{q(\partial)}\rightarrow V_{q(\partial)}, v\mapsto \phi(v)$$ is also a module homomorphism from $U_{q(\partial)}$ to $V_{q(\partial)}$.
Denote this functor by $\Phi_{q(\partial)}$. Then we have $$\Phi_{p(\partial)}\Phi_{q(\partial)}=\Phi_{p(\partial)+q(\partial)}\;\;\text{for any $p(\partial), q(\partial)\in\C[\partial]$,}$$
and $\Phi_{q(\partial)}$ is the identity functor if $q(\partial)\in\C$.
Hence, $\Phi_{q(\partial)}$ is an isomorphism from the category of generalized conformal modules to itself.
\end{rmk}
\begin{rmk}
By Proposition \ref{Pro-tw1}, we can obtain generalized conformal modules from conformal modules by the construction given above. Therefore, the theory of generalized conformal modules is much richer than that of conformal modules.
\end{rmk}

\begin{pro}\label{pro-2}\cite{CK}
Each non-trivial conformal module over $Vir$, which is free of rank one as a $\C[\partial]$-module, has the form
$$M_{a,b}=\C[\partial]v, \ L_\lambda v=(\partial+a \lambda +b)v$$
for some $a,b\in\C$.  $M_{a,b}$ is irreducible if and only if $a\neq 0$. Moreover, each finite non-trivial irreducible conformal module over $Vir$ is of rank one.
\end{pro}

By Propositions \ref{Pro-tw1} and \ref{pro-2}, ${M_{a,b}}_{q(\partial)}$ is a generalized conformal module over $Vir$ for each $q(\partial)\in \mathbb{C}[\partial]$. For convenience, denote ${M_{a,b}}_{q(\partial)}$ by $M_{q(\partial),a,b}$. Note that $M_{q(\partial),a,b}$ is irreducible if and only if $a\neq 0$. In Section 3, we will show that any non-trivial generalized conformal module over $Vir$, which is free of rank one as a $\C[\partial]$-module, is isomorphic to $M_{q(\partial),a,b}$ for some $a$, $b\in \mathbb{C}$ and $q(\partial)\in \mathbb{C}[\partial]$.

\begin{defi}\cite{K1}
Let $\mathcal{A}$ be a Lie conformal algebra. A $\text{Lie}( \mathcal{A})^+$-module $V$ is {\bf conformal} if for any $v \in V$ and $ a \in \mathcal{A}$, there exists $N \in \mathbb{Z}_+$ such that $a_{(n)}v=0$ when $n>N$. Modules of this kind are often called {\bf restricted} modules in reference.
A $\text{Lie}( \mathcal{A})^e$-module $V$ is called {\bf conformal} if it is conformal as a $\text{Lie}( \mathcal{A})^+$-module.
\end{defi}

There is a close connection between generalized conformal modules over Lie conformal algebras and modules over the corresponding extended annihilation Lie algebras.

\begin{pro}\label{pro-3}\cite[Remark 2.9a]{K1}, \cite[Proposition 2.1, Corollary 2.1]{CK}
Let $\mathcal{A}$ be a Lie conformal algebra. Then the following statements hold.
\begin{enumerate}
\item \label{a1}Let $V$ be a generalized conformal module (resp. conformal module) over $\mathcal{A}$. For any $a\in \mathcal{A}$ and $v\in V$, set
\begin{eqnarray*}
 a_\lambda v=\sum_{j\in \mathbb{Z}_+}(a_{(j)}v)\frac{\lambda^j}{j!}.
 \end{eqnarray*}
Then $V$ is a module (resp. conformal module) over $\text{Lie}( \mathcal{A})^e$ with the action:
\begin{eqnarray}
a_{(n)}. v=a_{(n)}v,\ \ \partial. v=\partial v\;\;\;\text{for any $a\in \mathcal{A}$, $n\in \Z_+$ and $v\in V$.}
\end{eqnarray}
\item \label{a2}If $V$ is a module (resp. conformal module) over $\text{Lie}( \mathcal{A})^e$, then $V$ is also a generalized conformal module (resp. conformal module) over $\mathcal{A}$ with the $\lambda$-action:
\begin{eqnarray}\label{eq-2}
 \partial v=\partial.v,\;\;a_\lambda v=\sum_{j\in \mathbb{Z}_+}(a_{(j)}.v)\frac{\lambda^j}{j!}\;\;\;\text{for any $a\in \mathcal{A}$ and $v\in V$.}
\end{eqnarray}
\item \label{a3}$V$ is an irreducible generalized conformal module (resp. conformal module) over $\mathcal{A}$ if and only if $V$ is also an irreducible module (resp. conformal module) over $\text{Lie}( \mathcal{A})^e$.
\end{enumerate}
\end{pro}
\subsection{Finite irreducible generalized conformal modules}
Assume that $V$ is a generalized conformal module over a Lie conformal algebra $\mathcal{A}$.
Set
 $$\text{Tor}~(V):=\{v\in V\,|\,f(\partial)v=0,\ \exists f(\partial)\in\C[\partial]\setminus\{0\}\}.$$
Next, we investigate the action of $\mathcal{A}$ on $\text{Tor}~(V)$. Referring to \cite{DK}, we have the following result.
\begin{pro}\cite{DK}
If $V$ is a conformal module over a Lie conformal algebra $\mathcal{A}$, then we have
$\mathcal{A}_\lambda (\text{Tor}~(V))=0$.
\end{pro}

For completeness, we present the proof.
For any $x\in \mathcal{A}, v\in \text{Tor}~(V)$,
write $x_\lambda v=\sum_i v_i \lambda^i$, where $v_i\in V$. Note that when $i>>0$, we have $v_i=0$.
If $x_\lambda v\neq 0$, then let $n:=\text{max}\{i\,|\,v_i\neq0\}$.
Since $v\in \text{Tor}~(V)$, there exists some nonzero $f(\partial)\in\C[\partial]$ such that $f(\partial)v=0$.
Denote $f(\partial)=\sum_j a_j \partial^j$ and let $m:=\text{max}\{j\,|\,a_j\neq 0\}$.
Then we have
$$0=x_\lambda (f(\partial)v)=f(\partial+\lambda)x_\lambda v=f(\partial+\lambda)\sum_i v_i \lambda^i.$$
Comparing the coefficient of $\lambda^{m+n}$, we have
$a_mv_n=0$, a contradiction. Therefore, $x_\lambda v=0$ for any $x\in \mathcal{A}$ and $v\in \text{Tor}~(V)$.

For a generalized conformal module, the above method may not hold. We present an example to show that {\bf for a generalized conformal module $V$ which is not conformal, $\mathcal{A}_\lambda (\text{Tor}~(V))$ may not be equal to $0$.}

\begin{ex}Let $\C[\partial]$ act on $\C[t]$ by
$\partial f(t)=\frac{d}{dt}f(t)=f'(t)$, where $f(t)\in \C[t]$. Then $\C[t]$ becomes a module over $\C[\partial]$.
Define $L_\lambda f(t)=e^{-\lambda t}f'(t)$ for any $f(t)\in \C[t]$.
One can check it directly that $\C[t]$ is a generalized conformal module over $Vir$. But
$\text{Tor}~(\C[t])=\C[t]$ and $L_\lambda \C[t]\neq 0$.
\end{ex}

For any $f(\partial)\in\C[\partial]$, define $V_f=\{v\in V\,|\,f(\partial)^nv=0\;\;\text{for some $n\in\Z_+$}\}$.
\begin{lem}\label{lem-tor1}
$V_f$ is a submodule of $V$.
\end{lem}
\begin{proof}
Let $v\in V$. Then there exists some $ n\in\Z_+$ such that $f(\partial)^nv=0$. Denote $g(\partial)=f(\partial)^n$. Set $x_\lambda v=\sum_{i=0}^\infty v_i\lambda^i$, where $v_i\in V$. Note that $g(\partial+\lambda)=\sum_{j=0}^\infty \frac{g^{(j)}(\partial)}{j!}\lambda^j$. Then we obtain
\begin{eqnarray*}
x_\lambda (g(\partial)v)&=&g(\partial+\lambda)x_\lambda v\\
&=&\sum_{i=0}^\infty\sum_{j=0}^\infty\frac{g^{(j)}(\partial)}{j!}v_i\lambda^{i+j}\\
&=&\sum_{m=0}^\infty\sum_{j=0}^m\frac{g^{(j)}(\partial)}{j!}v_{m-j}\lambda^{m}\\
&=& 0.
\end{eqnarray*}
Therefore, we have $\sum_{j=0}^m\frac{g^{(j)}(\partial)}{j!}v_{m-j}=0$ for each $m$. Then we get
\begin{eqnarray*}
g(\partial)v_m\in \mathbb{C}[\partial]v_0+\ldots+\mathbb{C}[\partial]v_{m-1}.
\end{eqnarray*}
Note that $g(\partial)v_0=0$. Therefore, we have $g(\partial)^{m+1}v_m=f(\partial)^{(m+1)n}v_0=0$ for each $m$. Consequently, we obtain $v_m\in V_f$ for each $m$. Then we get
$x_\lambda v\in V_f[[\lambda]]$ for any $x\in \mathcal{A}$. Then this conclusion holds.
\end{proof}
\begin{cor}
$\text{Tor}~(V)$ is a submodule of $V$.
\end{cor}
\begin{proof}It follows from Lemma \ref{lem-tor1} immediately.
\end{proof}
\begin{lem}\label{lem-tor2}If $V$ is an irreducible generalized conformal module over a Lie conformal algebra $\mathcal{A}$ and $\text{Tor}~(V)\neq 0$, then
$V=V_{\partial-\alpha}$ for some $\alpha\in\C$.
\end{lem}
\begin{proof}
If $\text{Tor}~(V)\neq 0$, then for any $v\in \text{Tor}~(V)$, there exists some nonzero $f(\partial)\in \C[\partial]$ such that $f(\partial)v=0$.
Since $\C$ is algebraically closed, $f(\partial)$ can be decomposed into a product of linear factors. So we can find some nonzero $u\in V$ and some
$\alpha\in\C$ such that $(\partial-\alpha)u=0$.
Thus, we have $V_{\partial-\alpha}\neq 0$, which is a submodule by Lemma \ref{lem-tor1}. Then the assertion follows.
\end{proof}


\begin{lem}\label{lem-tor3}
Let $V$ be a finite generalized conformal module  over a Lie conformal algebra $\mathcal{A}$.
Assume that $u\in V$ such that $(\partial-\alpha)u=0$ for some $\alpha\in\C$.
Then $\mathcal{A}_\lambda u=0$.
\end{lem}
\begin{proof}
\delete{Take $v\in Tor V$. There exists some nonzero $f(\partial)$ such that $f(\partial)v=0$.
Since $\C$ is algebraically closed, $f(\partial)$ can be decomposed into a product of linear factors. So By the discussion above, we can find some nonzero }
Denote $x_\lambda u=\sum_{i=0}^\infty u_i\lambda^i$, where $u_i\in V$ for each $i$. For convenience, denote $u_{-1}=0$.
Then we obtain
$$x_\lambda(\partial u)=(\partial+\lambda)x_\lambda u
=\sum_{i=0}^\infty (\partial u_i+u_{i-1})\lambda^i
=\sum_{i=0}^\infty (\alpha u_i)\lambda^i=\alpha x_\lambda u.$$
By comparing the coefficients of $\lambda^i$, we have $(\partial-\alpha)u_i=-u_{i-1}$ for each $i$.
We deduce that $(\partial-\alpha)^ju_m=(-1)^ju_{m-j}$ for any $j\leq m+1$.

Consider the ascending sequence of $\C[\partial]$-submodules
$$(u_0)\subseteq (u_0,u_1)\subseteq\cdots \subseteq (u_0,\cdots,u_k)\subseteq\cdots,$$
where $(u_0,\cdots,u_k)$ denotes the $\C[\partial]$-submodule generated by $u_0,\cdots,u_k$.
Since $\C[\partial]$ is a Noetherian ring, the above ascending sequence is stable, i.e.,
there exists some nonnegative integer $n$ such that $u_j\in (u_0,\cdots,u_n)$ for any $j$.
Since $(\partial-\alpha)^{n+1}(u_0,\cdots,u_n)=0$,
we deduce that $(\partial-\alpha)^{n+1}u_j=0$ for any $j$.
Then for any $i$, we have $$u_i=(\partial-\alpha)^{n+1}u_{n+1+i}=0.$$
Hence we get $x_\lambda u=0$.
\end{proof}

Let $\mathcal{A}$ be a Lie conformal algebra.
For any $\alpha\in\C$, denote by $\C_\alpha$ the one-dimensional trivial module where $\partial$ acts as the scalar $\alpha$, i.e.,
$$\C_\alpha=\C v,\;\;\; \partial v=\alpha v,\;\;\; \mathcal{A}_\lambda v=0.$$
\begin{lem}\label{lem-tor4}If $V$ is a finite irreducible generalized conformal module over a Lie conformal algebra $\mathcal{A}$ and $\text{Tor}~(V)\neq 0$. Then
$V\cong\C_\alpha$ for some $\alpha\in\C$.
\end{lem}
\begin{proof}
It follows from Lemmas \ref{lem-tor2} and \ref{lem-tor3} immediately.
\end{proof}
\begin{pro}\label{pro-finite}
Any finite irreducible generalized conformal module $V$ over a Lie conformal algebra $\mathcal{A}$
falls into one of the following two classes:
\begin{enumerate}
\item (trivial module) $\C_\alpha, \alpha\in\C$;
\item (non-trivial module) a free $\C[\partial]$-module of finite rank.
\end{enumerate}
\end{pro}
\begin{proof}
Note that $\C[\partial]$ is a principle ideal domain. Then $V$ as a $\C[\partial]$-module can be the direct sum of a free $\C[\partial]$-module of finite rank and $\text{Tor}~( V)$. Then this conclusion follows directly from Lemma \ref{lem-tor4}.
\delete{Let $V$ be a finite irreducible generalized conformal module of $\mathcal{A}$.
Define $N(V)=\{v\in V\,|\,\mathcal{A}_\lambda v=0\}$. Obviously, $N(V)$ is a submodule of $V$.
By Lemma \ref{lem-tor2}, we have $\text{Tor} V\subseteq N(V)$.
If $N(V)=0$, then $\text{Tor}~(V)= 0$, $V$ is of class (2).
If $N(V)\neq 0$, then $V=N(V)$ and $A_\lambda V=0$, i.e.,
$V$ only admits a structure of $\C[\partial]$-module.
Since $V$ is simple, $V\cong\C_a: \C_a=\C v, \partial v=a v, A_\lambda v=0$.}
\end{proof}

\begin{rmk}
By Proposition \ref{pro-finite}, in order to classify finite irreducible generalized conformal modules over a Lie conformal algebra $\mathcal{A}$, we only need to classify finite irreducible generalized conformal modules of class (b) over $\mathcal{A}$.
\end{rmk}

\section{Classification of non-trivial generalized conformal modules over $Vir$ which are free and of rank one as $\C[\partial]$-modules}
In this section, we give a classification of non-trivial generalized conformal modules over $Vir$ which are free and of rank one as $\C[\partial]$-modules and investigate the irreducibilities and isomorphism classes of these modules. Moreover, based on these results, we obtain a classification of modules over $W_1$ which are free of rank one over $\C[L_{-1}]$, which recovers a result of \cite{HCL}.

Assume that $V=\C[\partial]v$ is a non-trivial generalized conformal module over $Vir$, which is free and of rank one as a $\C[\partial]$-module.
Then we set $L_\lambda v=f(\partial, \lambda)v$, where $f(\partial, \lambda)\in \C[\partial][[\lambda]]$. By the definition of generalized conformal modules, we have
\begin{equation}\label{main}
(\lambda-\mu)f(\partial, \lambda+\mu)=f(\partial, \lambda)f(\partial+\lambda,\mu)
-f(\partial, \mu)f(\partial+\mu, \lambda).
\end{equation}
Therefore, the classification of all non-trivial generalized conformal module structures on $V$ over $Vir$ is equivalent to giving all nonzero solutions of Eq.~(\ref{main}) under the condition $f(\partial, \lambda)\in \C[\partial][[\lambda]]$. Note that by Proposition \ref{pro-2}, if $f(\partial, \lambda)\in \C[\partial][\lambda]\backslash\{0\}$, then $f(\partial, \lambda)=\partial+a \lambda+b$ for some $a$, $b\in \C$. Next, we need to find out all other nonzero solutions of Eq.~(\ref{main}) in $\C[\partial][[\lambda]]$.

Set $f(\partial, \lambda)=\sum_{i=0}^\infty f_i(\partial)\lambda^i$, where $f_i(\partial)\in \C[\partial]$ for each $i$.
Since $f(\partial, \lambda)\in \C[\partial][[\lambda]]$, this sum may be infinite, i.e., the degree of $\lambda$ in $f(\partial, \lambda)$ may be infinite. Moreover, although the degree of $\partial$ in $f_i(\partial)$ is finite for each $i$, they may not have
a common upper bound.

First, we solve the following equation
\begin{equation}\label{main2}
f(\partial, \lambda+\mu)=f(\partial, \lambda)f(\partial+\lambda,\mu),
\end{equation}
where $f(\partial, \lambda)\in \C[\partial][[\lambda]]$. Eq. (\ref{main2}) has a close connection to Eq. (\ref{main}).
\begin{lem}\label{lem-em3}
If $f(\partial, \lambda)$ is a solution of Eq.~(\ref{main}) and $g(\partial, \lambda)$ is a solution of Eq.~(\ref{main2}), then $f(\partial, \lambda)g(\partial, \lambda)
$ is a solution of Eq.~(\ref{main}).
\end{lem}
\begin{proof}
One can check it directly.
\end{proof}
For any $h(\lambda)=\sum_{n=0}^{\infty}a_i \lambda^i\in \C[[\lambda]]$,
denote its formal derivative by
$$h'(\lambda)=(h(\lambda))'=\sum_{n=1}^{\infty}ia_i \lambda^{i-1}$$
and its $n$-th formal derivative by
$h^{(n)}(\lambda)=(h^{(n-1)}(\lambda))'$.
Then
we have $h(\lambda+\mu)\in \C[[\lambda,\mu]]$ and the formal Taylor expansion holds.
\begin{lem}\label{lem-FT}
For any $h(\lambda)\in \C[[\lambda]]$, we have
\[
h(\lambda+\mu)=\sum_{n=0}^{\infty} h^{(n)}(\lambda)\mu^{[n]},
\]
where $\mu^{[n]}=\frac{ \mu^{n}}{n!}$.
\end{lem}
\begin{proof}
One can check it directly.
\end{proof}
Note that a power series $h(\lambda)\in \C[[\lambda]]$ is invertible if and only if $h(0)\neq0$.
\begin{lem}\label{lem-em1}
Assume that $h(\lambda)\in \C[[\lambda]]$ and $p(\lambda)\in \C[\lambda]$. If
\begin{eqnarray}
\label{eqq-m2}h'(\lambda)=h(\lambda)p(\lambda),
\end{eqnarray} then
$h(\lambda)=0$ or $h(\lambda)=e^{q(\lambda)}$ for some $q(\lambda)\in \C[\lambda]$ such that $q'(\lambda)=p(\lambda)$.
\end{lem}
\begin{proof}
Write $h(\lambda)=\sum_{i=0}^{\infty}a_i \lambda^i$, where $a_i\in \C$.
If $h(0)=a_0=0$, then by comparing the constant terms on both sides of Eq.~(\ref{eqq-m2}), we have $a_1=0$. Continuing this procedure, we deduce $h(\lambda)=0$.

If $a_0\neq 0$, take $q(\lambda)\in \C[\lambda]$ such that $q'(\lambda)=p(\lambda)$ and $e^{q(0)}=a_0$.
Clearly, $h_1(\lambda)=h(\lambda)-e^{q(\lambda)}$ also satisfies Eq.~ (\ref{eqq-m2}) and $h_1(0)=0$. Then $h_1(\lambda)=0$. Therefore, $h(\lambda)=e^{q(\lambda)}$.
\end{proof}

\begin{lem}\label{lem-em2}
Any nonzero solution of Eq.~(\ref{main2}) is  of the form $f(\partial, \lambda)=e^{q(\partial+\lambda)-q(\partial)}$ for some $q(\lambda)\in \C[\lambda]$.
\end{lem}
\begin{proof}
Assume that $f(\partial,\lambda)$ is a nonzero solution of Eq.~(\ref{main2}).
Setting $\partial=0$ in Eq.~(\ref{main2}) gives
\begin{equation}\label{main3}
f(0, \lambda+\mu)=f(0, \lambda)f(\lambda,\mu).
\end{equation}
Set $h(\lambda)=f(0, \lambda)$.  Denote $h(\lambda)=\sum_{n=0}^{\infty}a_i \lambda^i$, where $a_i\in \C$ for each $i\in \Z_+$.

First we claim $a_0\neq 0$.
Let $m={\rm min}\{i\,|\,a_i\neq 0\}$. If $m>0$, comparing coefficients of $\lambda^0\mu^m$ on both sides of $f(0, \lambda+\mu)=f(0, \lambda)f(\lambda,\mu)$ gives $a_m=0$, which contradicts with $a_m\neq 0$. Therefore, $a_0\neq 0$. Then the claim follows.

Since $a_0\neq 0$, $h(\lambda)$ is invertible in $\C[[\lambda]]$.
Therefore, by Lemma \ref{lem-FT}, we get
\[
f(\lambda,\mu)=h(\lambda)^{-1}h(\lambda+\mu)=\sum_{n=0}^{\infty} h(\lambda)^{-1}h^{(n)}(\lambda)\mu^{[n]}.
\]
Since
$f(\lambda,\mu)\in\C[\lambda][[\mu]]$,
we deduce $h(\lambda)^{-1}h^{(n)}(\lambda)\in \C[\lambda]$ for any $n\geq 0$.
But this condition is equivalent to
$h(\lambda)^{-1}h'(\lambda)\in \C[\lambda]$.
Set $h(\lambda)^{-1}h'(\lambda)=p(\lambda)\in \C[\lambda]$.
By Lemma \ref{lem-em1}, we deduce that $h(\lambda)=e^{q(\lambda)}$ with some $q(\lambda)\in \C[\lambda]$ and $q'(\lambda)=p(\lambda)$. Hence we have $f(\lambda, \mu)=h(\lambda)^{-1}h(\lambda+\mu)=e^{q(\lambda+\mu)-q(\mu)}$.
\end{proof}
\begin{pro}\label{prop1}If $f(\partial, \lambda)$ is a solution of Eq.~(\ref{main}) and $q(\partial)$ is an arbitrary  polynomial in $\C[\partial]$, then $f(\partial, \lambda)e^{q(\partial+\lambda)-q(\partial)}$ is also a solution of Eq.~(\ref{main}).
\end{pro}
\begin{proof}
It follows directly from Lemmas \ref{lem-em3} and \ref{lem-em2}.
\end{proof}
Suppose that $f(\partial, \lambda)\in \C[\partial][[\lambda]]$ is a nonzero solution
of Eq.~(\ref{main}).
Setting $\mu=0$ in Eq.~(\ref{main}) gives
$$\lambda f(\partial, \lambda)=f(\partial, \lambda)f(\partial+\lambda,0)
-f(\partial, 0)f(\partial, \lambda),$$
which implies $f(\partial, 0)=\partial+b$ for some $b\in \C$. Let $\widetilde{\partial}=\partial+b$. Then we set
\begin{eqnarray}
f(\partial, \lambda)=f(\widetilde{\partial}-b, \lambda)=g(\widetilde{\partial}, \lambda).
 \end{eqnarray}
 Obviously, $g(\widetilde{\partial}, 0)=\widetilde{\partial}$. Then Eq.~(\ref{main}) becomes
\begin{eqnarray}\label{main-1}
(\lambda-\mu)g(\widetilde{\partial}, \lambda+\mu)=g(\widetilde{\partial},\lambda)g(\widetilde{\partial}+\lambda, \mu)-g(\widetilde{\partial},\mu)g(\widetilde{\partial}+\mu, \lambda).
\end{eqnarray}
Set $g(\widetilde{\partial}, \lambda)=\sum_{i} \sum_ja_{i,j}\widetilde{\partial}^i \lambda^j$, where $a_{i,j}\in \C$.
The condition $g(\widetilde{\partial}, \lambda)\in \C[\widetilde{\partial}][[\lambda]]$ is equivalent to, for each $p$,
$a_{i,p}=0$ when $i>>0$.
Moreover, $g(\widetilde{\partial}, 0)=\widetilde{\partial}$ means $a_{i,0}=\delta_{i,1}, i\in \Z_+$.

\begin{lem}
Plugging $g(\widetilde{\partial}, \lambda)=\sum_i\sum_j a_{i,j}\widetilde{\partial}^i \lambda^j$ into Eq.~(\ref{main-1}), we have
\begin{equation}\label{e-1}
\aligned
&(\binom{p+r-1}{r}-\binom{p+r-1}{r-1})a_{i,p+r-1}\\
=&\sum_{k=0}^i\sum_{s=0}^p\binom{k+s}{k}a_{i-k,p-s}a_{k+s,r}-\sum_{k=0}^i\sum_{s=0}^r\binom{k+s}{k}a_{i-k, r-s}a_{k+s,p}.
\endaligned
\end{equation}
\end{lem}
\begin{proof}
Since $g(\widetilde{\partial}, \lambda)=\sum_i\sum_j a_{i,j}\widetilde{\partial}^i \lambda^j$, we have
\begin{equation*}
g(\widetilde{\partial}, \lambda+\mu)=\sum_j\sum_i a_{i,j}\widetilde{\partial}^i (\lambda+\mu)^j
=\sum_j\sum_i\sum_{k=0}^j\binom{j}{k} a_{i,j}\widetilde{\partial}^i\lambda^k\mu^{j-k}.
\end{equation*}
Consequently, we get
\begin{equation*}
(\lambda-\mu)g(\widetilde{\partial}, \lambda+\mu)=\sum_j\sum_i\sum_{k=0}^j \binom{j}{k} a_{i,j}\widetilde{\partial}^i(\lambda^{k+1}\mu^{j-k}-\lambda^k\mu^{j-k+1}).
\end{equation*}
The coefficient of $\widetilde{\partial}^i\lambda^p\mu^r$ in $(\lambda-\mu)g(\widetilde{\partial}, \lambda+\mu)$
is $$(\binom{p+r-1}{r}-\binom{p+r-1}{r-1})a_{i,p+r-1}.$$
In addition, we obtain
\begin{equation*}
g(\widetilde{\partial}, \lambda)g(\widetilde{\partial}+\lambda, \mu)
=\sum_{i_1,j_1,i_2,j_2, k} \binom{i_2}{k}a_{i_1,j_1}a_{i_2,j_2}\widetilde{\partial}^{i_1+k}\lambda^{j_1+i_2-k}\mu^{j_2}.
\end{equation*}
Setting $s=i_2-k$ and replacing the indices $i_1,i_2,j_1,j_2,k$ by $i_1,s,j_1,j_2,k$, we have
\begin{equation*}
g(\widetilde{\partial}, \lambda)g(\widetilde{\partial}+\lambda, \mu)
=\sum_{i_1,j_1,s,j_2, k} \binom{k+s}{k}a_{i_1,j_1}a_{k+s,j_2}\widetilde{\partial}^{i_1+k}\lambda^{j_1+s}\mu^{j_2}.
\end{equation*}
Then the coefficient of $\widetilde{\partial}^i\lambda^p\mu^r$ in $g(\widetilde{\partial}, \lambda)g(\widetilde{\partial}+\lambda, \mu)$
is $$\sum_{k=0}^i\sum_{s=0}^p\binom{k+s}{k}a_{i-k,p-s}a_{k+s,r}.$$
Similarly, the coefficient of $\widetilde{\partial}^i\lambda^p\mu^r$ in $g(\widetilde{\partial}, \mu)g(\widetilde{\partial}+\mu, \lambda)$
is $$\sum_{k=0}^i\sum_{s=0}^r\binom{k+s}{k}a_{i-k, r-s}a_{k+s,p}.$$
Then the desired assertion follows from Eq.~(\ref{main-1}) directly.\end{proof}

Next, we have the following recursive formula.
\begin{cor}\label{coro1}
If $g(\widetilde{\partial}, \lambda)=\sum_i\sum_j a_{i,j}\widetilde{\partial}^i \lambda^j$ is a nonzero solution of Eq.~(\ref{main-1}), then we have
\begin{equation}
(p+i-1)a_{i,p}
=\sum_{k=0}^i\sum_{s=1}^p\binom{k+s}{k}a_{i-k,p-s}a_{k+s,1}\;\;\text{for any $i$ and $p\in \Z_+$.}
\end{equation}
\end{cor}
\begin{proof}
Setting $r=1$ in Eq.~(\ref{e-1}) and noting $a_{i,0}=\delta_{i,1}$, we have
\begin{equation*}
\aligned
&(p-1)a_{i,p}\\
=&\sum_{k=0}^i\sum_{s=0}^p\binom{k+s}{k}a_{i-k,p-s}a_{k+s,1}-
\sum_{k=0}^i\binom{k+1}{k}a_{i-k, 0}a_{k+1,p}-
\sum_{k=0}^i\binom{k}{k}a_{i-k,1}a_{k,p}\\
=&\sum_{k=0}^i\sum_{s=1}^p\binom{k+s}{k}a_{i-k,p-s}a_{k+s, 1}-
ia_{i,p}.
\endaligned
\end{equation*}
Then the desired assertion follows.\end{proof}

\begin{rmk}
By Corollary \ref{coro1} and using induction on $p$, we can deduce
that each coefficient
$a_{i,p}$ is determined by $a_{j,0}, a_{k,1}, j,k\in\Z_+$.
\end{rmk}

\begin{cor}\label{coro2}
If $g(\widetilde{\partial}, \lambda)=\sum_i\sum_j a_{i,j}\widetilde{\partial}^i \lambda^j$ is a nonzero solution of Eq.~(\ref{main-1}) and $a_{k,1}=0$ for any $k\geq 1$, then we have
$a_{i,p}=0$ for any $p\geq 2$ and $g(\widetilde{\partial}, \lambda)=\widetilde{\partial}+a\lambda$, where $a=a_{0,1}$.
\end{cor}
\begin{proof}
If $p\geq 2$, then we have $p+i-1\neq 0$ for any $i\in \Z_+$.
Then by Corollary \ref{coro1}, we obtain
\begin{equation}
a_{i,p}
=\frac{1}{p+i-1}\sum_{k=0}^i\sum_{s=1}^p\binom{k+s}{k}a_{i-k,p-s}a_{k+s,1},
\end{equation}
where $k+s\geq 1$ on the right hand side.
Noting $a_{k,1}=0$ for any $k\geq 1$, we deduce $a_{i,p}=0$ for any $p\geq 2$.
Then we have $g(\widetilde{\partial}, \lambda)=\widetilde{\partial}+a_{0,1}\lambda$.
\end{proof}

\begin{lem}\label{thm-1}
Any nonzero solution of Eq.~(\ref{main}) is of the form
$$f(\partial, \lambda)=(\partial+a \lambda+b)e^{q(\partial+\lambda)-q(\partial)}$$ for some $a,b\in\C$ and some $q(\partial)\in\C[\partial]$.
\end{lem}
\begin{proof}
Suppose that $f(\partial, \lambda)\in \C[\partial][[\lambda]]$ is a nonzero solution
of Eq.~(\ref{main}). Note that $f(\partial, 0)=\partial+b$ for some $b\in \C$. Set $g(\widetilde{\partial}, \lambda)=f(\widetilde{\partial}-b, \lambda)=f(\partial, \lambda)$ with $\widetilde{\partial}=\partial+b$.
Note that $g(\widetilde{\partial},0)=\widetilde{\partial}$. Assume that $g(\widetilde{\partial}, \lambda)=\sum_i\sum_j a_{i,j}\widetilde{\partial}^i \lambda^j$ is a nonzero solution of Eq.~(\ref{main-1}).

By Corollary \ref{coro2}, if $a_{k,1}=0$ for any $k\geq 1$, then $g(\widetilde{\partial}, \lambda)=\widetilde{\partial}+a\lambda$ for some $a\in \C$. Then in this case, we obtain
\begin{eqnarray*}
f(\partial, \lambda)=g(\widetilde{\partial}, \lambda)=\partial+a\lambda+b.
\end{eqnarray*}
Take $q(\partial)=1\in \C[\partial]$. Then $f(\partial, \lambda)$ is of the desired form.

Next, we assume $a_{k,1}\neq 0$ with some $k\geq 1$.

Since $a_{k,1}=0$ when $k>>0$, there exists some $n\geq 1$ such that $a_{n,1}\neq 0$ and $a_{m,1}=0$ when $m>n$.
Then $g(\widetilde{\partial},\lambda)$ has the form
$$\widetilde{\partial}+(a_{0,1}+a_{1,1}\widetilde{\partial}+\cdots+a_{n,1}\widetilde{\partial}^n)\lambda+(\text{terms}\ \text{with}\ \text{deg}_\lambda\geq2).$$

Take $d_k=\frac{a_{k,1}}{k}$ for $k\geq 1$, and let $q(\widetilde{\partial})=d_n\widetilde{\partial}^n+\cdots+d_1\widetilde{\partial}\in\C[\widetilde{\partial}]$.
By direct computation, we have
$$q(\widetilde{\partial}+\lambda)-q(\widetilde{\partial})=(a_{n,1}\widetilde{\partial}^{n-1}+\cdots+a_{2,1}\widetilde{\partial}+a_{1,1})\lambda+\lambda^2\phi(\widetilde{\partial}, \lambda)$$
with $\phi(\widetilde{\partial}, \lambda)\in \C[\widetilde{\partial},\lambda]$.
Consequently, we obtain $$e^{-q(\widetilde{\partial}+\lambda)+q(\widetilde{\partial})}=1-(a_{n,1}\widetilde{\partial}^{n-1}+\cdots+a_{2,1}\widetilde{\partial}+a_{1,1})\lambda+\lambda^2h(\widetilde{\partial}, \lambda)$$
with $h(\widetilde{\partial}, \lambda)\in \C[\widetilde{\partial}][[\lambda]]\subset \C[[\widetilde{\partial},\lambda]]$.

By Proposition \ref{prop1}, $g(\widetilde{\partial}, \lambda)e^{-q(\widetilde{\partial}+\lambda)+q(\widetilde{\partial})}$ is also a nonzero solution of Eq.~(\ref{main-1}).
By direct computation, we get
$$g(\widetilde{\partial}, \lambda)e^{-q(\widetilde{\partial}+\lambda)+q(\widetilde{\partial})}=\widetilde{\partial}+a_{0,1}\lambda+(\text{terms}\ \text{with}\ \text{deg}_\lambda\geq2).$$
Then by Corollary \ref{coro2}, we have $$g(\widetilde{\partial}, \lambda)e^{-q(\widetilde{\partial}+\lambda)+q(\widetilde{\partial})}=\widetilde{\partial}+a_{0,1}\lambda.$$
Therefore, we obtain
\begin{eqnarray*}
g(\widetilde{\partial}, \lambda)=(\widetilde{\partial}+a_{0,1}\lambda)e^{q(\widetilde{\partial}+\lambda)-q(\widetilde{\partial})}.
\end{eqnarray*}
Consequently, setting $a=a_{0, 1}$, we get
\begin{eqnarray*}
f(\partial, \lambda)=g(\widetilde{\partial}, \lambda)=(\partial+a\lambda+b)e^{q(\partial+\lambda+b)-q(\partial+b)}.
\end{eqnarray*}
Note that $q(\partial+b)$ is a polynomial in $\C[\partial]$. Therefore, the conclusion holds.
\delete{Setting $a=a_{0, 1}$ and replacing $\partial$ by $\partial+b$ as we agree,
the desired assertion holds.}
\end{proof}

Finally, we present the classification result and the irreducibilities of generalized conformal modules over $Vir$ that are free and of rank one as $\C[\partial]$-modules.
\begin{thm}\label{thm-2}
Let $V=\C[\partial]v$ be free and of rank one as a $\C[\partial]$-module and $V$ be a non-trivial generalized conformal module over $Vir$. Then $V$ must be of the following form:
\begin{eqnarray}
L_\lambda v=e^{q(\lambda+\partial)-q(\partial)}(\partial+a\lambda+b)v,
\end{eqnarray}
for some $a$, $b\in \C$ and $q(\partial)\in \C[\partial]$. Denote it by $M_{q(\partial), a, b}$. Moreover, $M_{q(\partial), a, b}$ is irreducible if and only if $a\neq 0$.
\end{thm}
\begin{proof}
It follows directly by Lemma \ref{thm-1} and  Proposition \ref{Pro-tw1}.
\delete{By Theorem \ref{thm-1}, we only need to consider the irreducibility of $V_{q(\partial), a, b}$. Since $e^{q(\lambda+\partial)-q(\partial)}$ is invertible in $\C[\partial][[\lambda]]$, it is easy to see that $V_{q(\partial), a, b}$ is irreducible if and only if $V_{0, a,b}$ is irreducible. Note that $V_{0, a,b}$ is just $M_{a,b}$ in Proposition \ref{pro-2}. Then by Proposition \ref{pro-2}, $V_{q(\partial), a, b}$ is irreducible if and only if $a\neq 0$.}
\end{proof}
\delete{
\begin{rmk}
Note that $V_{q(\partial), a, b}$ is conformal if and only if $q(\partial)$ is a constant polynomial.
\end{rmk}
Finally, we investigate the irreducibilities of these modules over $Vir$.
\begin{pro}
Let $V$ be a module over a Lie conformal algebra $\mathcal {A}$, and $q(\partial)\in \C[\partial]$. Define a new $\lambda$-action of $\mathcal {A}$ on $V$ as follows:
\begin{eqnarray*}
\widetilde{a_\lambda v}=e^{q(\partial+\lambda)-q(\partial)}a_\lambda v,\;\;\text{for any $a\in \mathcal {A}$ and $v\in V$.}
\end{eqnarray*}
Then $V$ is also a module over $\mathcal {A}$ under this new $\lambda$-action. Denote this module by $V_{q(\partial)}$. We call that $V_{q(\partial)}$ is the {\bf twist of $V$ by $q(\partial)$}.
\end{pro}}

We now turn to the question of when $M_{q(\partial),a,b}$ is conformal. The following two lemmas will be needed.
\begin{lem}\label{lem-cm1}
For any $q(x)\in\C[x]$,
$e^{q(x+y)-q(x)}\in\C[x,y]$ if and only if $q(x)\in\C$.
\end{lem}
\begin{proof}If $q(x)\in\C$, then $e^{q(x+y)-q(x)}=1\in\C[x,y]$. Conversely, we assume $e^{q(x+y)-q(x)}\in\C[x,y]$.
Note that $e^{q(x+y)-q(x)}$, viewed as a complex function on $\C^2$, has no zeros. So it can not be a non-constant polynomial in $\C[x,y]$.
Hence if $e^{q(x+y)-q(x)}\in\C[x,y]$, it must be a constant.
Say $e^{q(x+y)-q(x)}=c$. 
Setting $x=0$, we have $c=e^{q(y)-q(0)}$. Then we deduce that $q(y)\in \C$.
\end{proof}

\begin{lem}\label{lem-cm2}For any $q(x)\in\C[x]$ and $a,b\in\C$,
$(x+ay+b)e^{q(x+y)-q(x)}\in\C[x,y]$ if and only if $e^{q(x+y)-q(x)}\in\C[x,y]$.
\end{lem}
\begin{proof}
The sufficient condition is obvious. We only show the necessary condition.
Let $I$ be an ideal of $\C[x,y]$.
Denote by  $\mathcal{Z}(I)$ the set of common zeros of all polynomials in $I$. It is a subset of $\C^2$.
For any subset $Z$ of $\C^2$, let $\mathcal{I}(Z)$ be the ideal of polynomials vanishing on $Z$. Then by Hilbert's Nullstellensatz, we have $\mathcal{I}(\mathcal{Z}(I))=\sqrt{I}$, where $\sqrt{I}$ is the radical ideal of $I$ (\cite{AM}).

If $(x+ay+b)e^{q(x+y)-q(x)}\in\C[x,y]$, we denote $I_1=((x+ay+b)e^{q(x+y)-q(x)})$ and $I_2=(x+ay+b)$. Since $\mathcal{Z}(I_1)=\mathcal{Z}(I_2)$, we have $\sqrt{I_1}=\sqrt{I_2}$. Since $x+ay+b$ is an irreducible polynomial of $\C[x,y]$, we have $\sqrt{I_2}=I_2$. Thus $(x+ay+b)e^{q(x+y)-q(x)}\in I_1\subseteq \sqrt{I_1}=\sqrt{I_2}=I_2$.
Therefore, $(x+ay+b)e^{q(x+y)-q(x)}=(x+ay+b)f(x,y)$ for some $f(x,y)\in\C[x,y]\subseteq \C[x][[y]]$. Since $\C[x][[y]]$ is a domain, we deduce $e^{q(x+y)-q(x)}=f(x,y)\in \C[x,y]$.
\end{proof}
\begin{pro}\label{pro-cm}
$M_{q(\partial),a,b}$ is conformal if and only if $q(\partial)\in\C$.
\end{pro}
\begin{proof}Recall that \begin{eqnarray}
M_{q(\partial),a,b}=\C[\partial] v, \
L_\lambda v=e^{q(\lambda+\partial)-q(\partial)}(\partial+a\lambda+b)v.
\end{eqnarray}
Thus $M_{q(\partial),a,b}$ is conformal if and only if $e^{q(\lambda+\partial)-q(\partial)}(\partial+a\lambda+b)\in\C[\partial,\lambda]$. Then the assertion follows directly from Lemmas \ref{lem-cm1} and \ref{lem-cm2}.
\end{proof}

\begin{pro}Let $q_1(\partial), q_2(\partial)\in\C[\partial]$ and $a_1,a_2,b_1,b_2\in\C$. Then
$M_{q_1(\partial), a_1,b_1}\cong M_{q_2(\partial), a_2,b_2}$ if and only if $$q_1(\partial)-q_2(\partial)\in\C,\;\; a_1=a_2,\;\;b_1=b_2.$$
\end{pro}
\begin{proof}
If $q_1(\partial)-q_2(\partial)\in\C, a_1=a_2,b_1=b_2$, then we have
$$e^{q_1(\lambda+\partial)-q_1(\partial)}(\partial+a_1\lambda+b_1)=
e^{q_2(\lambda+\partial)-q_2(\partial)}(\partial+a_2\lambda+b_2)$$ and consequently $M_{q_1(\partial), a_1,b_1}\cong M_{q_2(\partial), a_2,b_2}$.
Conversely, if $M_{q_1(\partial), a_1,b_1}\cong M_{q_2(\partial), a_2,b_2}$, then $$M_{a_1,b_1}\cong(M_{q_1(\partial), a_1,b_1})_{-q_1(\partial)}\cong (M_{q_2(\partial), a_2,b_2})_{-q_1(\partial)}\cong M_{q_2(\partial)-q_1(\partial), a_2,b_2}$$ is a conformal module (see Proposition \ref{Pro-tw1} and Remark \ref{Rmk-tw}). By Proposition \ref{pro-cm}, we have $q_1(\partial)-q_2(\partial)\in\C$. Note that two conformal modules $M_{a_1,b_1}$ and $M_{a_2, b_2}$ are isomorphic if and only if $a_1=a_2, b_1=b_2$. Then this conclusion follows.
\end{proof}

Finally, as an application of Theorem \ref{thm-2}, we will present a classification result of modules over $W_1$ which are free of rank one as $\C[L_{-1}]$-modules.

\begin{lem}\label{lem-w1}
Let $q(\partial)\in \C[\partial]$. Then for each $n\in \Z_+$, we obtain
\begin{eqnarray}
{\frac{\partial^n}{\partial\lambda^n}e^{q(\lambda+\partial)-q(\partial)}}|_{\lambda=0}=B_n(q'(\partial), q''(\partial), \ldots, q^{(n)}(\partial)),
\end{eqnarray}
where $B_n$ is the $n$-th Bell polynomial, i.e.,
\begin{eqnarray*}
B_n(x_1,x_2,\cdots, x_n)=\sum_{p_1+2p_2+\cdots+np_n=n}\frac{n!}{p_1!p_2!\cdots p_n!}(\frac{x_1}{1!})^{p_1}(\frac{x_2}{2!})^{p_2}\cdots (\frac{x_n}{n!})^{p_n}.
\end{eqnarray*}
\end{lem}
\begin{proof}
By \cite[Section 3.3]{C}, we obtain
\begin{eqnarray}\label{eq-Bell}
e^{\sum_{n=1}^\infty\frac{x_n}{n!}\lambda^n}=\sum_{n=0}^\infty\frac{B_n(x_1, x_2, \ldots,x_n)}{n!}\lambda^n,
\end{eqnarray}
where $x_1$, $x_2$, $\ldots$, $x_n$, $\ldots$ are variables with $x_ix_j=x_jx_i$ for each $i$, $j$.
By Taylor expansion, we get
\begin{eqnarray}
q(\lambda+\partial)-q(\partial)=\sum_{n=1}^\infty\frac{q^{(n)}(\partial)}{n!}\lambda^n.
\end{eqnarray}
Therefore, by Eq.~(\ref{eq-Bell}), we have
\begin{eqnarray*}
e^{q(\lambda+\partial)-q(\partial)}&=&e^{\sum_{n=1}^\infty\frac{q^{(n)}(\partial)}{n!}\lambda^n}\\
&=&\sum_{n=0}^\infty\frac{B_n(q'(\partial), q''(\partial),\ldots,q^{(n)}(\partial))}{n!}\lambda^n.
\end{eqnarray*}
Note that $e^{q(\lambda+\partial)-q(\partial)}=\sum_{n=0}^\infty
\frac{\partial^n}{\partial\lambda^n}e^{q(\lambda+\partial)-q(\partial)}|_{\lambda=0}\frac{\lambda^n}{n!}$.
Then the conclusion holds.
\end{proof}
\begin{thm}\label{thm-3}
Let $M=\C[L_{-1}]v$ be free and of rank one as a $\C[L_{-1}]$-module and $M$ be a non-trivial module over $W_1$. Then the module action of $W_1$ on $M$ must be of the following form:
\begin{eqnarray*}
L_{j}. L_{-1}^nv&=&\sum_{k=0}^{j+1}\binom{n}{k}\frac{(j+1)!}{(j+1-k)!}L_{-1}^{n-k}((j+1-k)B_{j-k}(q'(L_{-1}), q''(L_{-1}),\ldots, q^{(j-k)}(L_{-1}))av\\
&&\quad+B_{j+1-k}(q'(L_{-1}), q''(L_{-1}),\ldots, q^{(j+1-k)}(L_{-1}))L_{-1} v)\;\;\text{ for any $j\in \Z_{\geq -1}$ and $n\in \Z_+$,}
\end{eqnarray*}
where $q(L_{-1})\in \C[L_{-1}]$ and $a\in \C$. Denote this module by $M_{q(\partial), a}$. 
\end{thm}
\begin{proof}
By Corollary \ref{cor-4}, a module $M$ over $W_1$ is equivalent to a module $M_0$ over $\text{Lie}(Vir)^e$, where $\partial v=L_{-1}v=L_{(0)}.v$. Therefore,  $M_0=\C[L_{-1}]v=\C[\partial]v$ is a module over $\text{Lie}(Vir)^e$. By Proposition \ref{pro-3}, $M_0$ is a generalized conformal module over $Vir$, which is free and of rank one as a $\C[\partial]$-module. By Theorem \ref{thm-2},  $M_0$ is of the form $M_{q(\partial), a, b}$ for some $q(\partial)\in \C[\partial]$ and $a$, $b\in \C$.

By Lemma \ref{lem-w1}, we obtain
\begin{eqnarray*}
e^{q(\lambda+\partial)-q(\partial)}=\sum_{i=0}^\infty B_i(q'(\partial), q''(\partial), \ldots, q^{(i)}(\partial))\lambda^{[i]}.
\end{eqnarray*}
Therefore, for each $n\in \Z_+$, we get
\begin{eqnarray*}
L_\lambda \partial^nv&=&(\partial+\lambda)^nL_\lambda v\\
&=&(\partial+\lambda)^ne^{q(\lambda+\partial)-q(\partial)}(\partial+a\lambda+b)v\\
&=&(\partial+\lambda)^n(\sum_{i=0}^\infty B_i(q'(\partial), q''(\partial), \ldots, q^{(i)}(\partial))\lambda^{[i]})(\partial+a\lambda+b)v.
\end{eqnarray*}
Write $L_\lambda \partial^nv=\sum_{i=0}^\infty (L_{(j)}\partial^nv)\lambda^{[j]}$. Comparing the coefficients of $\lambda^{[j]}$ on both sides, we obtain
\begin{eqnarray}
L_{(j)}. \partial^nv&=&\sum_{k=0}^j\binom{n}{k}\frac{j!}{(j-k)!}\partial^{n-k}((j-k)B_{j-k-1}(q'(\partial), q''(\partial),\ldots, q^{(j-k-1)}(\partial))av\nonumber\\
&&\quad +B_{j-k}(q'(\partial), q''(\partial),\ldots, q^{(j-k)}(\partial))(\partial+b)v)\;\; \text{for any $j$, $n\in \Z_{+}$.}\label{eq-act1}
\end{eqnarray}
By Proposition \ref{pro-3} and the definition of $M_0$, the module action of $\text{Lie}(Vir)^e$ on $M_0$ is given by
Eq.~(\ref{eq-act1}) with $b=0$ and $\partial v=L_{-1}v$. Then by Corollary \ref{cor-4}, we obtain that $M$ must be of the form $M_{q(\partial), a}$ for some $q(\partial)\in \C[\partial]$ and $a\in \C$.
\delete{
Note that the irreducibility of $M_{q(\partial), a}$ over $W_1$ is the same as that of $V_{q(\partial), a, 0}$ over $Vir$. Therefore, by Theorem \ref{thm-2}, $M_{q(\partial), a}$ is irreducible if and only if $a\neq 0$.

Let $\varphi:M_{q_1(\partial), a_1}\rightarrow M_{q_2(\partial), a_2}$ be an isomorphism of $W_1$-modules. Since
$\varphi(L_{-1}^i. v)=L_{-1}^i.\varphi(v)$ for any $i\in \Z_+$, and $\varphi$ is an isomorphism of vector spaces, we get $\varphi(v)=\alpha v$ for some $\alpha\in \C\backslash \{0\}$. Note that for $M_{q_1(\partial), a_1}$, we obtain
\delete{\begin{eqnarray*}
L_{j}. v=(j+1)B_{j}(q'_1(L_{-1}), q''_1(L_{-1}),\cdots, q_1^{(j)}(L_{-1}))av+B_{j+1}(q'_1(L_{-1}), q''_1(L_{-1}),\cdots, q_1^{(j+1)}(L_{-1}))L_{-1} v.
\end{eqnarray*}
Therefore, we get}
\begin{eqnarray*}
L_0. v=a_1v+q'_1(L_{-1})L_{-1}v.
\end{eqnarray*}
Therefore, we get
\begin{eqnarray*}
\varphi(L_0.v)&=&\varphi(a_1v+q'_1(L_{-1})L_{-1}v)=\alpha(av+q'_1(L_{-1})L_{-1}v)\\
&=&L_0.\varphi(v)=\alpha(a_2v+q'_2(L_{-1})L_{-1}v).
\end{eqnarray*}
Consequently, we have $a_1=a_2$ and $q'_1(L_{-1})=q'_2(L_{-1})$. Moreover, if $q'_1(\partial)=q'_2(\partial)$ and $a_1=a_2$, it is obvious that $M_{q_1(\partial), a_1}$ is isomorphic to $M_{q_2(\partial), a_2}$. Thus, this conclusion holds.}
\end{proof}

\delete{\begin{cor}
Let $M=\C[L_{-1}]v$ be free of rank one as a $\C[L_{-1}]$-module and $M$ be a non-trivial conformal module over $W_1$, i.e., there exists an integer $N\in \Z_+$ such that $L_n.u=0$ for any $n> N$ and $u\in M$. Then the module action of $W_1$ on $V$ must be of the following form:
\begin{eqnarray}\label{eq-1}
&L_{j}.(L_{-1}^nv)=(j+1)!(\binom{n}{j+1}L_{-1}^{n-j-1}L_{-1}v+\binom{n}{j}L_{-1}^{n-j}av),\;\;\text{for each $j\in \Z_{\geq-1}$ and $n\in \Z_+$.}
\end{eqnarray}
where $a\in \C$. Denote this module by $M_{a}$. Moreover, $M_{a}$ is irreducible if and only if $a\neq 0$.
\end{cor}
\begin{proof}
It is straightforward by Proposition \ref{pro-2} and Theorem \ref{thm-3}.
\end{proof}}

\begin{rmk}
A classification of modules over $W_1$ which are free of rank one as $\mathbb C[L_{-1}]$-modules has previously been obtained in \cite[Theorem 4.1]{HCL} via a direct Lie-algebraic approach. It can be verified by a straightforward comparison that our classification in Theorem 3.14 coincides with theirs, under the parameter correspondence
\[
a = 1-b, \qquad q'(x) = \alpha'(x),
\]
where $b$ and $\alpha(x)$ are the parameters in \cite[Theorem 4.1]{HCL}.
However, our approach via generalized conformal modules over $Vir$ provides a new conceptual framework for these modules. In particular, it reveals that the polynomial $q(\partial)$ arises naturally from the twisted construction of generalized conformal modules given in Proposition \ref{Pro-tw1}, and it places the classification within the broader context of the representation theory of Lie conformal algebras. This perspective not only unifies the existing result but also suggests possible generalizations to higher-rank modules and to other Lie conformal algebras.
\end{rmk}
\delete{Finally, we investigate the isomorphism classes of $M_{q(\partial), a}$.

\begin{thm}
$M_{q_1(\partial), a_1}$ is isomorphic to $M_{q_2(\partial), a_2}$ if and only if $q'_1(\partial)=q'_2(\partial)$ and $a_1=a_2$.
\end{thm}
\begin{proof}
Let $\varphi:M_{q_1(\partial), a_1}\rightarrow M_{q_2(\partial), a_2}$ be an isomorphism of $W_1$-modules. Since
$\varphi(L_{-1}^i. v)=L_{-1}^i.\varphi(v)$ for any $i\in \Z_+$, and $\varphi$ is an isomorphism of vector spaces, we get $\varphi(v)=\alpha v$ for some $\alpha\in \C\backslash \{0\}$. Note that for $M_{q_1(\partial), a_1}$, we obtain
\delete{\begin{eqnarray*}
L_{j}. v=(j+1)B_{j}(q'_1(L_{-1}), q''_1(L_{-1}),\cdots, q_1^{(j)}(L_{-1}))av+B_{j+1}(q'_1(L_{-1}), q''_1(L_{-1}),\cdots, q_1^{(j+1)}(L_{-1}))L_{-1} v.
\end{eqnarray*}
Therefore, we get}
\begin{eqnarray*}
L_0. v=a_1v+q'_1(L_{-1})L_{-1}v.
\end{eqnarray*}
Therefore, we get
\begin{eqnarray*}
\varphi(L_0.v)&=&\varphi(a_1v+q'_1(L_{-1})L_{-1}v)=\alpha(av+q'_1(L_{-1})L_{-1}v)\\
&=&L_0.\varphi(v)=\alpha(a_2v+q'_2(L_{-1})L_{-1}v).
\end{eqnarray*}
Consequently, we have $a_1=a_2$ and $q'_1(L_{-1})=q'_2(L_{-1})$. Moreover, if $q'_1(\partial)=q'_2(\partial)$ and $a_1=a_2$, it is obvious that $M_{q_1(\partial), a_1}$ is isomorphic to $M_{q_2(\partial), a_2}$. Thus, this conclusion holds.
\end{proof}}
\section{Generalized conformal modules of higher ranks over $Vir$}
\delete{Note that we have obtained a classification of generalized conformal modules of rank one over $Vir$ and presented a criterion on their
irreducibilities. In the conformal case, it was shown in \cite{CK} that any finite non-trivial irreducible conformal module over $Vir$ is of rank one .
One expects that this also holds for the general case. But this is false, unfortunately. We will construct a class of generalized conformal modules $V_{a, b, c(\partial)}$ of arbitrary finite rank $n$ over $Vir$, parametrized by $a$, $b\in C$ and a matrix $C(\partial)\in\mathbb C[\partial]^{n\times n}$,
 and investigate the irreducibilities and isomorphism classes of these modules. Moveover, we will present a concrete example to show that there exist irreducible generalized conformal modules of arbitrary ranks over $Vir$, which reveals how different the generalized conformal modules are from the conformal modules.
Undoubtedly, the theory of generalized conformal modules over Lie conformal algebras is very rich and interesting.
We hope our work will promote the development of this theory.}
In this section, we construct a class of generalized conformal modules $V_{a,b,C(\partial)}$ of arbitrary finite rank $n$ over $Vir$, parametrized by $a,b\in\mathbb C$ and $C(\partial)\in\mathbb C[\partial]^{n\times n}$, investigate their irreducibilities and isomorphism classes, and characterize when these modules are conformal. We show that, in contrast to the conformal case where every finite non-trivial irreducible module is of rank one \cite{CK}, irreducible generalized conformal modules of arbitrary finite rank over $Vir$ exist.
\delete{
We start with an example of generalized conformal modules of rank two over $Vir$.
\begin{ex}
Let $V=\C[\partial]v_1\oplus \C[\partial]v_2$ be a free $\C[\partial]$-module of rank two with the $\C[\partial]$-basis $v_1,v_2$. For any $a,b\in\C$,
define $$L_{\lambda}(v_1,v_2)=(v_1,v_2)\begin{pmatrix}(\partial+a\lambda+b) cos\, \lambda &(\partial+a\lambda+b) sin\, \lambda\\
-(\partial+a\lambda+b) sin\, \lambda &(\partial+a\lambda+b) cos\, \lambda
\end{pmatrix}.$$
Note that $sin\, \lambda=\sum_{n=0}^\infty \frac{(-1)^nx^{2n+1}}{(2n+1)!}\in \C[[\lambda]]$ and $cos\, \lambda=\sum_{n=0}^\infty \frac{(-1)^nx^{2n}}{(2n)!}\in \C[[\lambda]].$
Then one can check it directly that $V$ is a generalized conformal module of rank two over $Vir$. Let $w_1=v_1+iv_2$ and $w_2=v_1-iv_2$. Note that $L_\lambda w_1=(\partial+a\lambda+b)e^{i\lambda}w_1$ and $L_\lambda w_2=(\partial+a\lambda+b)e^{-i\lambda}w_2$. Therefore, $W_1=\C[\partial]w_1$ and $W_2=\C[\partial]w_2$ are submodules of $V$. Consequently, $V$ is not irreducible and $V=W_1\oplus W_1$ as generalized conformal modules of $Vir$.
\end{ex}}
\delete{\begin{rmk}
Note
$$\begin{pmatrix}(\partial+a\lambda+b) cos\, \lambda, &(\partial+a\lambda+b) sin\, \lambda\\
-(\partial+a\lambda+b) sin\, \lambda, &(\partial+a\lambda+b) cos\, \lambda
\end{pmatrix}=(\partial+a\lambda+b)\begin{pmatrix} cos\, \lambda, &sin\, \lambda\\
-sin\, \lambda, &cos\, \lambda
\end{pmatrix}$$
and
\begin{equation*}
\aligned
&sin\, \lambda=\sum_{n=0}^\infty \frac{(-1)^nx^{2n+1}}{(2n+1)!}\in \C[[\lambda]], \\
&cos\, \lambda=\sum_{n=0}^\infty \frac{(-1)^nx^{2n}}{(2n)!}\in \C[[\lambda]].
\endaligned
\end{equation*}
\end{rmk}}

Let $V$ be a generalized conformal module over $Vir$, which is free and of rank $n$ as a $\C[\partial]$-module.
Assume that $v_1,\ldots,v_n$ is a $\C[\partial]$-basis of $V$ and
$$L_\lambda v_i=\sum_{j=1}^n f_{j,i}(\partial,\lambda)v_j,$$
where $f_{j,i}(\partial,\lambda)\in \C[\partial][[\lambda]]$.
Set
$$A(\partial,\lambda)=(f_{j,i}(\partial,\lambda))\in (\C[\partial][[\lambda]])^{n\times n},$$
 which determines the module action of $Vir$ on $V$. \delete{We call $A(\partial,\lambda)$ the structure matrix of $V$ (with respect to $v_1,\cdots,v_n$).}
Then the condition $(\lambda-\mu)L_{\lambda+\mu}=L_\lambda L_\mu-L_\mu L_\lambda$ can be
reformulated as
\begin{equation}\label{e0}
(\lambda-\mu)A(\partial, \lambda+\mu)=A(\partial, \lambda)A(\partial+\lambda, \mu)
-A(\partial, \mu)A(\partial+\mu, \lambda).
\end{equation}
Solving Eq.~(\ref{e0}) in $(\C[\partial][\lambda])^{n\times n}$ is equivalent
to classifying conformal modules of rank $n$ over $Vir$, 
while
solving Eq.~(\ref{e0}) in $(\C[\partial][[\lambda]])^{n\times n}$ is equivalent
to classifying generalized conformal modules of rank $n$ over $Vir$.

\delete{Note $\partial-L_{(0)}$ lies in the center of $Lie(Vir)^e$, by Schur'Lemma $\partial-L_{(0)}$
acts as a scalar on $V$. Say $L_{(0)}=\partial+b$ for some $b\in \C$.
This is equivalent to $A(\partial,0)=(\partial+b)E$.
Replacing $\partial+b$ by $\partial$, one can assume $A(\partial,0)=\partial E$.}

\begin{lem}\label{lem-1}
Let $A(\partial, \lambda)$ be a solution of Eq.~(\ref{e0}). If there exists $B(\partial)=(b_{ij}(\partial))\in \C[[\partial]]^{n\times n}$ such that $B(\partial)^{-1}A(\partial,\lambda)B(\partial+\lambda)\in(\C[\partial][[\lambda]])^{n\times n}$, then
$$B(\partial)^{-1}A(\partial,\lambda)B(\partial+\lambda)$$
is also a solution of Eq.~(\ref{e0}). \delete{where
$B(\partial)=(b_{ij}(\partial))$ and $B(\partial+\lambda)=(b_{ij}(\partial+\lambda))$.}
\end{lem}
\begin{proof}
It is straightforward.
\end{proof}
If $B(\partial)^{-1}$ commutes with $A(\partial,\lambda)$, then we have
$$B(\partial)^{-1}A(\partial,\lambda)B(\partial+\lambda)=A(\partial,\lambda)B(\partial)^{-1}B(\partial+\lambda).$$
Note that $$B(\partial+\lambda)=\sum_{k=0}^\infty B^{(k)}(\partial)\frac{\lambda^k}{k!},$$ where
$B(\partial):=(b_{ij}(\partial))\in \C[[\partial]]^{n\times n}$ and $B^{(k)}(\partial)=(b_{ij}^{(k)}(\partial))$.
For our purpose, we need a criterion on whether $B(\partial)^{-1}B(\partial+\lambda)\in(\C[\partial][[\lambda]])^{n\times n}$.

\begin{lem}\label{lem-BC}If $C(\partial):=B(\partial)^{-1}B'(\partial)\in \C[\partial]^{n\times n}$, then
\begin{equation}\label{e1}
B(\partial)^{-1}B^{(k)}(\partial)=(\frac{d}{d\partial}+C(\partial))^kI_n\in \C[\partial]^{n\times n}
\end{equation}
for any $k\in\Z_+$ and $$B(\partial)^{-1}B(\partial+\lambda)=\sum_{k=0}^\infty (\frac{d}{d\partial}+C(\partial))^kI_n\frac{\lambda^k}{k!}\in
(\C[\partial][[\lambda]])^{n\times n}.$$
\end{lem}
\begin{proof}
Denote $C_k(\partial)=B(\partial)^{-1}B^{(k)}(\partial), k\in\Z_+$.
We prove this lemma by induction on $k$. 
The case for $k=0$ is obvious.
Assume $C_k(\partial)=(\frac{d}{d\partial}+C(\partial))^kI_n\in \C[\partial]^{n\times n}$. Since
$B^{(k)}(\partial)=B(\partial)C_k(\partial)$, then we have
\begin{equation*}
\aligned
B^{(k+1)}(\partial)
=(B(\partial)C_k(\partial))'
=B(\partial)C'_k(\partial)+B'(\partial)C_k(\partial)
=B(\partial)(C'_k(\partial)+C(\partial)C_k(\partial)).
\endaligned
\end{equation*}
Hence we get
$$C_{k+1}(\partial)=C'_k(\partial)+C(\partial)C_k(\partial)
=(\frac{d}{d\partial}+C(\partial))C_k(\partial)
=(\frac{d}{d\partial}+C(\partial))^{k+1}I_n\in \C[\partial]^{n\times n}.$$
Then this conclusion holds.
\end{proof}

\delete{As in the proof of Lemma \ref{lem-BC}, we obtain a sequence of matrices $\{C_k(\partial)~|~ k\in\Z_+\}$ satisfying the recursive relations
\begin{equation}\label{e1}
C_{k+1}(\partial)=C'_k(\partial)+C(\partial)C_k(\partial),
\end{equation}
where $C_0(\partial)=I_n$. Then we deduce  $$C_{k}(\partial)=(\frac{d}{d\partial}+C(\partial))^kI_n,$$ indicating that the sequence of matrices $\{C_k(\partial)~|~ k\in\Z_+\}$ is completely determined by a single matrix $C(\partial)\in (\C[\partial])^{n\times n}$.
Then we have
$$B(\partial)^{-1}B(\partial+\lambda)=\sum_{k=0}^\infty C_k(\partial)\frac{\lambda^k}{k!}.$$
Consequently, we get
$$A(\partial,\lambda)B(\partial)^{-1}B(\partial+\lambda)=A(\partial,\lambda)\sum_{k=0}^\infty C_k(\partial)\frac{\lambda^k}{k!}.$$}

\begin{rmk}When we compute $B(\partial)^{-1}A(\partial,\lambda)B(\partial+\lambda)$ in the case of $A(\partial,\lambda)$ commuting with $B(\partial)^{-1}$, we do not need the explicit expression of $B(\partial)$. 
What really works is $C(\partial)=B(\partial)^{-1}B'(\partial)$.
\end{rmk}

In the sequel, let $C(\partial)\in \C[\partial]^{n\times n}$. Denote $\mathfrak{L}=\frac{d}{d\partial}+C(\partial)$, which is a linear operator on $\C[\partial]^{n\times n}$. Then
we have the following result.
\begin{pro}\label{pro-gen1}
For any $a$, $b\in \C$,
$$(\partial+a \lambda+b)\sum_{k=0}^\infty (\mathfrak{L}^kI_n)\frac{\lambda^k}{k!}$$
is a solution of Eq.~(\ref{e0}).
\end{pro}
\begin{proof}
Note that $A(\partial, \lambda)=(\partial+a\lambda+b)I_n$ is a solution of Eq.~(\ref{e0}) and $(\partial+a\lambda+b)I_n$ commutes with each matrix in $\C[[\partial]]^{n\times n}$. Then the assertion follows from Lemmas \ref{lem-1} and \ref{lem-BC} immediately. 
\end{proof}
By Proposition \ref{pro-gen1}, $V=\C[\partial]v_1\oplus \C[\partial]v_2\oplus \cdots \oplus \C[\partial]v_n$ is a generalized conformal module over $Vir$ defined by
\begin{eqnarray*}
\aligned
L_\lambda (v_1,\ldots, v_n)&=(v_1,\ldots, v_n)(\partial+a \lambda+b)\sum_{k=0}^\infty (\mathfrak{L}^kI_n)\frac{\lambda^k}{k!}=(v_1,\ldots, v_n)(\partial+a \lambda+b)(e^{\lambda\mathfrak{L}}I_n)\\
&=(v_1,\ldots, v_n)\sum_{k=0}^\infty ((\partial+b)(\mathfrak{L}^kI_n)+ka(\mathfrak{L}^{k-1}I_n))\frac{\lambda^k}{k!}
.
\endaligned
\end{eqnarray*}
Denote this generalized conformal module by $V_{a,b,C(\partial)}$.
\begin{ex}
Take $C(\partial)=C\in\C^{n\times n}$, i.e., a complex matrix.
Then by Eq.~(\ref{e1}), we have $\mathfrak{L}^kI_n=C^k$ for each $k\in \Z_+$. Therefore, by Proposition \ref{pro-gen1},
$$(\partial+a \lambda+b)\sum_{k=0}^\infty C^k\frac{\lambda^k}{k!}=(\partial+a \lambda+b)e^{\lambda C}$$
is a solution of Eq.~(\ref{e0}). Therefore, $V_{a,b,C} = \mathbb{C}[\partial]v_1\oplus \mathbb{C}[\partial]v_2\oplus \cdots \oplus \mathbb{C}[\partial]v_n$ is a generalized conformal module over $Vir$ defined by
\begin{eqnarray*}
L_\lambda (v_1, v_2, \ldots, v_n)=(v_1, v_2, \ldots, v_n)(\partial + a\lambda + b)e^{\lambda C}.
\end{eqnarray*}
By the fundamental theorem of algebra, the matrix $C$ has at least one eigenvalue $\mu \in \mathbb{C}$ and a corresponding nonzero eigenvector
$\xi \in \mathbb{C}^n$, i.e., $C \xi = \mu \xi$.  Then one can check that $U=\C[\partial] ((v_1, v_2, \ldots, v_n)\xi)$ is a submodule of $V$. Therefore, if $n>1$, then $V_{a,b,C}$ is not irreducible.

In particular, if we take $n=2$ and $C=
\begin{pmatrix}0&1\\
-1&0
\end{pmatrix}$, then we have $$(\partial+a \lambda+b)e^{\lambda C}=(\partial+a \lambda+b)\begin{pmatrix} \text{cos}\, \lambda &\text{sin}\, \lambda\\
-\text{sin}\, \lambda &\text{cos}\, \lambda
\end{pmatrix}.$$
Let $w_1=v_1+iv_2$ and $w_2=v_1-iv_2$. Note that $L_\lambda w_1=(\partial+a\lambda+b)e^{i\lambda}w_1$ and $L_\lambda w_2=(\partial+a\lambda+b)e^{-i\lambda}w_2$. Therefore, $W_1=\C[\partial]w_1$ and $W_2=\C[\partial]w_2$ are submodules of $V_{a,b,C}$. Consequently, we have $V_{a,b,C}=W_1\oplus W_2$ as the direct sum of generalized conformal modules over $Vir$.
\end{ex}

Next, we will discuss the irreducibility of $V_{a,b,C(\partial)}$.

\begin{pro}\label{pro-irr1}
Let $C(\partial)\in \C[\partial]^{n\times n}$ and $b\in \C$. If $a=0$, then the generalized conformal module $V_{0,b,C(\partial)}$ is reducible.
\end{pro}
\begin{proof}
  $(\partial+b)V_{0,b,C(\partial)}\cong V_{1,b,C(\partial)}$ is a proper submodule of $V_{0,b,C(\partial)}$. Therefore, $V_{0,b,C(\partial)}$ is reducible.
\end{proof}
\begin{rmk}
Note that $V_{0,b,C(\partial)}/(\partial+b)V_{0,b,C(\partial)}$ is isomorphic to the direct sum of $n$ copies of $\C_b$.
\end{rmk}

By Proposition \ref{pro-irr1}, we only need to consider the irreducibility of $V_{a,b,C(\partial)}$ when $a\neq 0$.
It is more convenient to realize $V_{a,b,C(\partial)}$ as
$$V_{a,b,C(\partial)}\cong \C[\partial]^n,\ \ f_1(\partial)v_1+\cdots+f_n(\partial)v_n\mapsto \begin{pmatrix}f_1(\partial)\\
\vdots\\
f_n(\partial)
\end{pmatrix}.$$
Then $V_{a,b,C(\partial)}$ is a generalized conformal module over $Vir$ defined by
\begin{eqnarray*}
L_\lambda \begin{pmatrix}a_1\\
\vdots\\
a_n
\end{pmatrix}=(\partial+a \lambda+b)(e^{\lambda\mathfrak{L}}I_n)
\begin{pmatrix}a_1\\
\vdots\\
a_n
\end{pmatrix}\;\; \text{for any $\begin{pmatrix}a_1\\
\vdots\\
a_n
\end{pmatrix}\in\C^n$}.
\end{eqnarray*}
Denote $\mathcal{L}=\frac{d}{d\partial}+C(\partial)$, which is a linear operator on $\C[\partial]^{n}$. Viewing $\partial$ as an operator on $\C[\partial]^{n}$, we have
$\mathcal{L}^k\partial=\partial \mathcal{L}^k+k\mathcal{L}^{k-1}$ for each $k\in\Z_+$.
Then we can deduce that $e^{\lambda \mathcal{L}}\partial=(\partial+\lambda)e^{\lambda \mathcal{L}}$. Then by
$$(e^{\lambda\mathfrak{L}}I_n)
\begin{pmatrix}a_1\\
\vdots\\
a_n
\end{pmatrix}=e^{\lambda\mathcal{L}}
\begin{pmatrix}a_1\\
\vdots\\
a_n
\end{pmatrix}\;\; \text{for any $\begin{pmatrix}a_1\\
\vdots\\
a_n
\end{pmatrix}\in\C^n$,}
$$
we deduce that
\delete{$$L_\lambda\begin{pmatrix}f_1(\partial)\\
\vdots\\
f_n(\partial)
\end{pmatrix}=(\partial+a\lambda+b)e^{\lambda\mathfrak{L}}\begin{pmatrix}f_1(\partial)\\
\vdots\\
f_n(\partial)
\end{pmatrix},\ \forall
\begin{pmatrix}f_1(\partial)\\
\vdots\\
f_n(\partial)
\end{pmatrix}\in V_{C(\partial)}
.$$
or }
\begin{equation}\label{V(x)}
L_\lambda\mathbf{f}(\partial)
=(\partial+a\lambda+b)e^{\lambda\mathcal{L}}\mathbf{f}(\partial)
=\sum_{k=0}^\infty\frac{\lambda^k}{k!} ((\partial+b)\mathcal{L}^k+ka\mathcal{L}^{k-1})\mathbf{f}(\partial),
\end{equation}
where $\mathbf{f}(\partial)=
(f_1(\partial),\cdots,f_n(\partial))^T\in V_{a,b,C(\partial)}$.

\begin{rmk}When $n=1$, take $A(\partial,\lambda)=\partial+a\lambda+b$ and $B(\partial)=e^{q(\partial)}$. Then $B(\partial)^{-1}A(\partial,\lambda)B(\partial+\lambda)
=(\partial+a\lambda+b)e^{q(\partial+\lambda)-q(\partial)}$, which induces the rank one module $M_{q(\partial), a,b}$ in Theorem \ref{thm-2}. In this case, we have $C(\partial)=B(\partial)^{-1}B'(\partial)=q'(\partial)$.
Moreover, if we take $b=0$, we actually obtain a classification of $W_1$-modules which are
free and of rank one as a $\C[L_{-1}]$-module. Therefore, those $W_1$-modules in Theorem \ref{thm-3} can also be described in the following forms:
$$M_{a,0,q(x)}=\C[x],\ L_kf(x)=(x\mathcal{L}^{k+1}+(k+1) a\mathcal{L}^{k})f(x),$$
where $f(x)\in \C[x]$ and $\mathcal{L}=\frac{d}{dx}+q'(x)$ is an operator on $\C[x]$.
\end{rmk}

We view $V_{a, b, C(\partial)}$ as a module over the Lie algebra $\text{Lie}( Vir)^e$ and thus a module over the Lie algebra $W_1$.
To distinguish, we use $x$ instead of $\partial$ as the variable and denote this module by
$V=V_{a, b, C(x)}$ for short.
Then by Propositions \ref{pro-1} and \ref{pro-3}, we have
$$V=\C[x]^n,\ L_k\mathbf{f}(x)=((x+b)\mathcal{L}^{k+1}+(k+1) a\mathcal{L}^{k})\mathbf{f}(x),$$
where $\mathbf{f}(x)=
(f_1(x),\cdots,f_n(x))^T\in V$ and $\mathcal{L}=\frac{d}{dx}+C(x)$ is an operator on $\C[x]^n$.
\delete{
Realize $V=\C[x]^n$ and define
$$L=\frac{d}{dx}+C(x):\C[x]^n\rightarrow \C[x]^n,\ (d_1(x),\ldots, d_n(x))^T\mapsto (d_1'(x),\ldots, d_n'(x))^T+C(x)(d_1(x),\ldots, d_n(x))^T.$$ Then
$$L_i(d_1(x),\ldots, d_n(x))^T=((x+b)L^{i+1}+(i+1)aL^i)(d_1(x),\ldots, d_n(x))^T, i\geq -1.$$}

Let $K=\C(x)$ be the fraction field of $\C[x]$ and $V_K=\C(x)\otimes_{\C[x]} V\cong \C(x)^n$. Since $\frac{d}{dx}$ extends uniquely to be a derivation of $\C(x)$, then
$\mathcal{L}$ and $L_k, k\geq -1$ extend to be acting on $V_K$ consequently.
\begin{lem}\label{lem-4.3}
For any $k\in\Z_+$, $v\in V$ and $f(x)\in \C(x)$, we have
$\mathcal{L}^kf(x)v=\sum_{i=0}^k \binom{k}{i}f^{(i)}(x)\mathcal{L}^{k-i}v$.
\end{lem}
\begin{proof}
Note that $\mathcal{L}f(x)v=f'(x)v+f(x)\mathcal{L}v$.
We prove the assertion by induction on $k$. The case for $k=0$ is obvious.
Assume $\mathcal{L}^kf(x)v=\sum_{i=0}^k \binom{k}{i}f^{(i)}(x)\mathcal{L}^{k-i}v$. Then we obtain $$\mathcal{L}^{k+1}f(x)v=\mathcal{L}\mathcal{L}^kf(x)v
=\sum_{i=0}^k \binom{k}{i}(f^{(i+1)}(x)\mathcal{L}^{k-i}+f^{(i)}(x)\mathcal{L}^{k+1-i})v
=\sum_{i=0}^{k+1} \binom{k+1}{i}f^{(i)}(x)\mathcal{L}^{k+1-i}v.$$
\end{proof}
\begin{lem}\label{lem-4.4-1}
For any $v\in V$ and $p\in \C(x)$, we have
\begin{align}
L_0(pv)&=pL_0v+(x+b)p' v,\label{eq-L0}\\
L_1(pv)&=pL_1v+(x+b)p'' v+2(x+b)p'\mathcal{L}v+2ap'v,\label{eq-L1}\\
L_2(pv)&=pL_2v+(x+b)p'''v+3(x+b)p''\mathcal{L}v+3ap''v+3(x+b)p'\mathcal{L}^2 v+6a p'\mathcal{L}v.\label{eq-L2}
\end{align}
\end{lem}
\begin{proof}Note that
$L_i=(x+b)\mathcal{L}^{i+1}+(i+1)a\mathcal{L}^i$ for any $i\geq -1$. Then the assertion follows from Lemma \ref{lem-4.3} directly.
\end{proof}
\delete{
Setting $p(x)=(x+b)^m$ in
Eq. (\ref{eq-L1}), we have
\begin{equation*}
L_1\bigl((x+b)^m v\bigr) = (x+b)^m L_1 v + 2m (x+b)^m L v + m(m-1+2a)(x+b)^{m-1} v,
\end{equation*} and
\begin{equation}
\aligned
L_2\bigl((x+b)^m v\bigr)&=(x+b)^m L_2 v+3m (x+b)^m L^2 v
+ 3m(m-1+2a)(x+b)^{m-1} L v\\
&\quad + m(m-1)(m-2+3a)(x+b)^{m-2} v
\endaligned
\end{equation}
Then
\begin{equation}
\aligned
L_2\bigl((x+b)^{m+1} v\bigr)&=(x+b)^{m+1} L_2 v+3(m+1) (x+b)^{m+1} L^2 v
+ 3(m+1)(m+2a)(x+b)^{m} L v\\
&\quad + (m+1)m(m-1+3a)(x+b)^{m-1} v
\endaligned
\end{equation}}

\begin{lem}\label{lem-4.4}
If
$X$ is a $W_1$-submodule of $V$, then $X_K$ is an $\mathcal{L}$-invariant subspace of $V_K$.
\end{lem}
\begin{proof}
For any $f(x)\in \C(x), v\in X$, by Lemma \ref{lem-4.3}, we have
$$\mathcal{L}(f(x)v)=f'(x)v+f(x)\mathcal{L}v=f'(x)v+f(x)(x+b)^{-1}(L_0-a)v\in X_K.$$
Then the assertion follows.
\end{proof}
\begin{rmk}
Since $L_i=(x+b)\mathcal{L}^{i+1}+(i+1)a\mathcal{L}^i$, by Lemma \ref{lem-4.4}, we have $L_i(X_K)\subseteq X_K$ for any $i\geq -1$.
\end{rmk}
\begin{pro}\label{thm-4.1}
If $a\neq 0$,
then $V$ is an irreducible module over $W_1$ if and only if 
$\mathcal{L}=\frac{d}{dx}+C(x)$ acts irreducibly on $V_K$.
\end{pro}
\begin{proof} Assume that $V=V_{a,b,C(x)}$ is an irreducible module over $W_1$.
Let $U$ be a nonzero $K$-subspace of $V_K$ which is invariant under the action of $\mathcal{L}$. Note that $\mathcal{L}V\subseteq V$.
Then $U\cap V$ is a $\C[x]$-submodule and invariant under the action of $\mathcal{L}$.
Since $L_i=(x+b)\mathcal{L}^{i+1}+(i+1)a\mathcal{L}^i$, $U\cap V$ is a $W_1$-submodule. Since $U\neq 0$, we have $U\cap V\neq 0$. Consequently, we obtain $U\cap V=V$. Thus we have $U=V_K$ and $\mathcal{L}$ acts irreducibly on $V_K$.

Conversely, assume that $\mathcal{L}$ acts irreducibly on $V_K$ and $X$ is a nonzero $W_1$-submodule of $V$. Then $X_K$ is a nonzero subspace of $V_K$.
If $X_K=V_K$, noting that $V$ is of finite rank over $\C[x]$, then there exists a nonzero polynomial $f(x)\in\C[x]$ such that $f(x)V\subseteq X$.
Assume that $p(x)$ is such a monic polynomial with the lowest degree.
For any $v\in V$, by Eq. (\ref{eq-L0}), we have $L_0(p(x)v)=(x+b)p'(x)v+p(x)L_0v$, which implies $(x+b)p'(x)v\in X$. Then we obtain $p(x)|(x+b)p'(x)$, i.e., $p(x)=(x+b)^m$ for some $m\in\Z_+$. Consequently, we have $(x+b)^mV\subseteq X$. If $m\geq1$, for any $v\in V$,
by setting $p(x)=(x+b)^m$ in
Eq. (\ref{eq-L1}), we have
\begin{equation*}
L_1\bigl((x+b)^m v\bigr) = (x+b)^m L_1 v + 2m (x+b)^m\mathcal{L}v + m(m-1+2a)(x+b)^{m-1} v.
\end{equation*}
By setting $p(x)=(x+b)^{m+1}$ in
Eq. (\ref{eq-L2}), we have
\begin{equation*}
\aligned
L_2\bigl((x+b)^{m+1} v\bigr)&=(x+b)^{m+1} L_2 v+3(m+1) (x+b)^{m+1}\mathcal{L}^2 v
+ 3(m+1)(m+2a)(x+b)^{m} \mathcal{L} v\\
&\quad + (m+1)m(m-1+3a)(x+b)^{m-1} v.
\endaligned
\end{equation*}
Since $a\neq 0$, we have $m-1+2a\neq0$ or $m-1+3a\neq 0$. Thus we deduce $(x+b)^{m-1}V\subseteq X$, which contradicts the choice of $p(x)$. Hence, we obtain $m=0$ and
$X=V$. Then this assertion follows.
\end{proof}
\begin{rmk}\label{rmk-4.2}
Note that the operator $\mathcal{L}$ is independent of the parameters $a$ and $b$. Therefore,  by Theorem \ref{thm-4.1}, the irreducibility of $V_{a,b,C(\partial)}$
is also independent of $a$ and $b$, provided that $a\neq 0$. It is an interesting phenomenon.
\end{rmk}
\begin{pro}\label{pro-4.1}
The operator $\mathcal{L}=\frac{d}{dx}+C(x)$ acts irreducibly on $V_K$ if and only if any nonzero $W_1$-submodule of $V$ is of rank $n$ as a $\C[x]$-module.
\end{pro}
\begin{proof}
If $X$ is a nonzero $W_1$-submodule of $V$, then $X_K$ is an $\mathcal{L}$-invariant subspace of $V_K$ by Lemma \ref{lem-4.4}.
Since $\mathcal{L}$ acts irreducibly on $V_K$, then $X_K=V_K$.
Thus $rank_{\mathbb{C}[x]}X=dim_{\mathbb{C}(x)}X_K=n$.

Conversely, if $U$ is a nonzero subspace of $V_K$ which is invariant under the action of $\mathcal{L}$, then $U\cap V$ is a nonzero
$W_1$-submodule. Then we have $rank_{\C[x]}(U\cap V)=n$.
Since $(U\cap V)_K\subseteq U$ and $dim_{\mathbb{C}(x)}(U\cap V)_K=rank_{\C[x]}(U\cap V)=n$, we obtain $dim_{\C(x)}U=n$ and $U=V_K$.
\end{proof}
We use the language of linear system of differential equations to get an criterion on whether $V_{a,b,C(x)}$ is irreducible.
\begin{pro}\label{pro-4.2-1}
The operator $\mathcal{L}=\frac{d}{dx}+C(x)$ acts irreducibly on $V_K$ if and only if for any $1\leq r<n$ and $B(x)\in K^{r\times r}$,
$Y'(x)+C(x)Y(x)=Y(x)B(x)$ has no nonzero solutions, where $Y(x)\in K^{n\times r}$.
\end{pro}
\begin{proof}
\delete{A proper subspace $U$ of $V_K$ has a $K$-basis $u_1,\cdots,u_r$, $1\leq r<n$.
Let $Y(x)=(v_1,\cdots,v_r)\in K^{n\times r}$.\yy{Let $Y(x)=(u_1,\cdots,u_r)\in K^{n\times r}$?}
Then $\mathcal{L}U\subseteq U$ if and only if $Y'(x)+C(x)Y(x)=Y(x)B(x)$ for some $B(x)\in K^{r\times r}$. The conclusion follows.
\yy{I rewrite this proof as follows.}}
Suppose that $\mathcal{L}$  acts irreducibly on $V_K$. Assume that there exist $1\leq r<n$, $B(x)\in K^{r\times r}$ and some nonzero $Y(x)\in K^{n\times r}$ such that $Y'(x)+C(x)Y(x)=Y(x)B(x)$. Let $U$ be the $K$-subspace spanned by the $r$ column vectors of $Y(x)$. Since $Y(x)\neq 0$, we have $U\neq 0$. Because $1\leq r<n$, $U$ is a proper subspace of $V_K$. For any $v \in U$, there exists $a(x) \in K^r$ such that $v = Y(x)a(x)$. Then we obtain
\begin{eqnarray*}
\mathcal{L}(v) &=& (Y(x)a(x))' + C(x)(Y(x)a(x)) = Y(x)'a(x) + Y(x)a'(x) + C(x) Y(x)a(x)\\
&=&(Y'(x) + C(x)Y(x))a(x) + Y(x)a'(x) = (Y(x)B(x))a(x) + Y(x)a'(x)\\
& =& Y(x)(B(x)a(x) + a'(x)) \in U.
\end{eqnarray*}
Thus $U$ is a nonzero proper $\mathcal{L}$-invariant subspace, a contradiction. Hence no such nonzero solution exists.

Conversely, suppose that the equation has no nonzero solution for any $1 \le r < n$ and any $B(x)$. Assume that there exists a nonzero proper $K$-subspace $U \subsetneq V_K$ which is invariant under the action of $\mathcal{L}$. Let $\dim_K U = r$. So we have $1 \le r < n$. Choose a $K$-basis $u_1,\dots,u_r$ of $U$ and let $Y(x) = (u_1,\dots,u_r) \in K^{n \times r}$. Note that $Y(x)\neq 0$. Since $U$ is $\mathcal L$-invariant, for each $j$, $\mathcal{L}(u_j) \in U$. Hence there exists a unique matrix $B = (b_{ij}(x)) \in K^{r \times r}$ such that $\mathcal{L}(u_j) = \sum_{i=1}^r b_{ij}(x) u_i$ which is equivalent to the matrix identity $\mathcal{L}(Y(x)) = Y(x)B(x)$, i.e., $Y'(x) + C(x)Y(x) = Y(x)B(x)$.
This gives a nonzero solution, contradicting the assumption. Therefore, $\mathcal{L}$  acts irreducibly on $V_K$.

\end{proof}

\begin{thm}\label{thm-c}Let $a,b\in\C$ and $C(\partial)\in\C[\partial]^{n\times n}$.
The generalized conformal module
$V_{a,b,C(\partial)}$ over $Vir$ is irreducible if and only if $a\neq 0$ and
$\frac{d}{dx}+C(\partial)$ acts irreducibly on $\C(\partial)^n$. Moreover, the following statements are equivalent:
\begin{enumerate}
\item $\frac{d}{dx}+C(\partial)$ acts irreducibly on $\C(\partial)^n$;
\item For any $1\leq r<n$ and $B(x)\in K^{r\times r}$,
$Y'(x)+C(x)Y(x)=Y(x)B(x)$ has no nonzero solutions, where $Y(x)\in K^{n\times r}$;
\item Any nonzero $W_1$-submodule of $V_{a,b,C(\partial)}$ is of rank $n$ as a $\C[x]$-module.
\end{enumerate}
\end{thm}
\begin{proof}
It follows directly by Propositions \ref{pro-irr1}, \ref{thm-4.1},  \ref{pro-4.1} and \ref{pro-4.2-1}.
\end{proof}

The following proposition characterizes the isomorphism classes of the modules $V_{a,b,C(\partial)}$.

\begin{pro}\label{iso-1}
$V_{a_1,b_1,C_1(\partial)}\cong V_{a_2,b_2,C_2(\partial)}$ if and only if $(a_1,b_1)=(a_2,b_2)$ and there exists $G(\partial)\in\C[\partial]^{n\times n}$ such that $\text{det}~G(\partial)\in\C\setminus \{0\}$ and $G'(\partial)=G(\partial)C_1(\partial)-C_2(\partial)G(\partial)$.
\end{pro}
\begin{proof}\delete{Recall that for any $a, b\in\C$ and $C(\partial)\in\C[\partial]^{n\times n}$, we have
\begin{equation*}
V_{a,b,C(\partial)}=\C[\partial]^n,\
L_\lambda\mathbf{f}(\partial)
=(\partial+a\lambda+b)e^{\lambda\mathcal{L}}\mathbf{f}(\partial)
=\sum_{k=0}^\infty\frac{\lambda^k}{k!} ((\partial+b)\mathcal{L}^k+ka\mathcal{L}^{k-1})\mathbf{f}(\partial),
\end{equation*}
where $\mathbf{f}(\partial)=
(f_1(\partial),\cdots,f_n(\partial))^T\in V_{a, b,C(\partial)}$.} Let $\mathcal{L}_i=\frac{d}{d\partial}+C_i(\partial)$ for $i=1$, $2$.
If $G'(\partial)=G(\partial)C_1(\partial)-C_2(\partial)G(\partial)$, then we obtain
$$\mathcal{L}_2G(\partial)=(\frac{d}{d\partial}+C_2(\partial))G(\partial)=G(\partial)\frac{d}{d\partial}+G'(\partial)+C_2(\partial)G(\partial)
=G(\partial)(\frac{d}{d\partial}+C_1(\partial))=G(\partial)\mathcal{L}_1.$$
For $G(\partial)\in\C[\partial]^{n\times n}$ with
$\text{det}~G(\partial)\in\C\setminus\{0\}$,
$\phi: V_{a,b,C_1(\partial)}\rightarrow V_{a,b,C_2(\partial)},\ \mathbf{f}(\partial)\mapsto G(\partial)\mathbf{f}(\partial)$ is an isomorphism of $\C[\partial]$-modules. Moreover, we have
$$\phi (L_\lambda\mathbf{f}(\partial))=
G(\partial)(\partial+a\lambda+b)e^{\lambda\mathcal{L}_1}\mathbf{f}(\partial)=
(\partial+a\lambda+b)e^{\lambda\mathcal{L}_2}G(\partial)\mathbf{f}(\partial)=
L_\lambda\phi(\mathbf{f}(\partial)).
$$ Therefore, we get $V_{a,b,C_1(\partial)}\cong V_{a,b,C_2(\partial)}$.

Conversely, if $V_{a_1,b_1,C_1(\partial)}\cong V_{a_2,b_2,C_2(\partial)}$,  assume that
$$\phi: V_{a_1,b_1,C_1(\partial)}\rightarrow V_{a_2,b_2,C_2(\partial)},\ \mathbf{f}(\partial)\mapsto G(\partial)\mathbf{f}(\partial),\ G(\partial)\in\C[\partial]^{n\times n},\ \text{det}~G(\partial)\in\C\setminus \{0\}$$ is an isomorphism. Viewing them as modules over the Lie algebra $\text{Lie}( Vir)^e\cong W_1\oplus \C(\partial-L_{-1})$, we have $\phi L_k=L_k \phi$ for any $k\geq -1$. Since $\phi L_{-1}=L_{-1}\phi$ and $L_{-1}$ acts by $\partial+b$ on $V_{a,b,C(\partial)}$, we deduce that $b_1=b_2$. Since $\phi L_0=L_0\phi$, we have
\begin{equation*}
\aligned
&(\partial+b) G(\partial)\frac{d}{d\partial}+ (\partial+b)G(\partial)C_1(\partial)+a_1G(\partial)
=G(\partial)((\partial+b)(\frac{d}{d\partial}+C_1(\partial))+a_1))\\
=&((\partial+b)(\frac{d}{d\partial}+C_2(\partial))+a_2))G(\partial)
=(\partial+b) G(\partial)\frac{d}{d\partial}+(\partial+b)(G'(\partial)+C_2(\partial)G(\partial))+a_2G(\partial),
\endaligned
\end{equation*}
which gives $(\partial+b) G(\partial)C_1(\partial)+a_1G(\partial)=(\partial+b)(G'(\partial)+C_2(\partial)G(\partial))+a_2G(\partial)$. Then we have $$(\partial+b)(G(\partial)C_1(\partial)G^{-1}(\partial)-(G'(\partial)+C_2(\partial)G(\partial))G(\partial)^{-1})=(a_2-a_1)I_n.$$
Since $G(\partial)^{-1}\in\C[\partial]^{n\times n}$, we deduce that $a_1=a_2$ and $G'(\partial)=G(\partial)C_1(\partial)-C_2(\partial)G(\partial)$.
\end{proof}

\begin{pro}Let $a,b\in\C$ and $C(\partial)\in\C[\partial]^{n\times n}$.
Then $V_{a,b,C(\partial)}$ is conformal if and only if $G'(\partial)+C(\partial)G(\partial)=0$ for some $G(\partial)\in\C[\partial]^{n\times n}$ with $\text{det}~G(\partial)\in\C\setminus\{0\}$.
\end{pro}
\begin{proof}
If $C(\partial)G(\partial)+G'(\partial)=0$ for some $G(\partial)\in\C[\partial]^{n\times n}$ with $\text{det}~G(\partial)\in\C\setminus\{0\}$, then $V_{a,b,C(\partial)}\cong V_{a,b,0}$ by Proposition \ref{iso-1}. Obviously, $V_{a,b,0}$ is conformal.

Conversely, if $V_{a,b,C(\partial)}$ is conformal, then by Eq. (\ref{V(x)}), we have
$\sum_{k=0}^\infty((\partial+b)(\mathfrak{L}^kI_n)+ka(\mathfrak{L}^{k-1}I_n))\frac{\lambda^k}{k!}$ $
\in(\C[\partial][\lambda])^{n\times n}$. Then there exists $ s\in\Z_+$ such that $(\partial+b)(\mathfrak{L}^kI_n)+ka(\mathfrak{L}^{k-1}I_n)=0$ for any $ k\geq s$.
\delete{
For arbitrary $k\geq s$, by $$\mathfrak{L}((\partial+b)(\mathfrak{L}^kI_n)+ka(\mathfrak{L}^{k-1}I_n))
=(\partial+b)(\mathfrak{L}^{k+1}I_n)+(ka+1)(\mathfrak{L}^{k}I_n)=0,$$ we deduce that $\mathfrak{L}^{k}I_n=0$ if $a\neq 1$. If $a=1$, then $(\partial+b)(\mathfrak{L}^kI_n)+k(\mathfrak{L}^{k-1}I_n)=0, k\geq s$.}
Note that $\mathfrak{L}^kI_n\in\C[\partial]^{n\times n}$.
By comparing the degree of $\partial$ in $(\partial+b)(\mathfrak{L}^kI_n)=-ka(\mathfrak{L}^{k-1}I_n)$, we deduce that
$\mathfrak{L}^kI_n=0$ for some $k\in\Z_+$. Assume that $\mathfrak{L}^mI_n=0$ while
$\mathfrak{L}^{m-1}I_n\neq0$.
Let $G(\partial)=\sum_{i=0}^{m-1}(-1)^i\frac{(\partial+b)^i}{i!}\mathfrak{L}^iI_n$. Then a direct computation gives
$$G'(\partial)+C(\partial)G(\partial)=\mathfrak{L}G(\partial)=
\mathfrak{L}\sum_{i=0}^{m-1}(-1)^i\frac{(\partial+b)^i}{i!}\mathfrak{L}^iI_n=0.$$
Since $G(-b)=I_n$, we have $\text{det}~G(\partial)\neq 0$. Then by Liouville's formula, we have $(\text{det}~G(\partial))'=-tr(C(\partial))~\text{det}~G(\partial)$. By comparing the degree of $\partial$ on both sides, we have $(\text{det}~G(\partial))'=0$. Therefore, $\text{det}~G(\partial)\in\C\setminus\{0\}$ and the assertion follows.
\end{proof}

For $n=2$, the criterion in Theorem \ref{thm-c} can be sharpened significantly: instead of solving a matrix equation, one only needs to check whether a certain differential equation has a rational solution.
\delete{
In this case, a proper $\mathcal{L}$-invariant subspace must be of dimension one. Denote $Y(x)=(u(x),v(x))^T\neq 0$ and
$C(x)=\begin{pmatrix}
c_{11}(x)&c_{12}(x)\\
c_{21}(x)&c_{22}(x)
\end{pmatrix}\in\C[x]^{n\times n}$. Consider the equation $$\mathcal{L}\begin{pmatrix}u(x)\\v(x)\end{pmatrix}=\begin{pmatrix}u'(x)\\v'(x)\end{pmatrix}
+C(x)\begin{pmatrix}u(x)\\v(x)\end{pmatrix}=
\alpha(x)\begin{pmatrix}u(x)\\v(x)\end{pmatrix}$$ with $\alpha(x)\in K$.
Then we have
$$\begin{cases}
u'(x)+c_{11}(x)u(x)+c_{12}(x)v(x)=\alpha(x)u(x)\\
v'(x)+c_{21}(x)u(x)+c_{22}(x)v(x)=\alpha(x)v(x)
\end{cases}.
$$
If $v(x)\neq 0$, denote $\gamma(x)=u(x)/v(x)$.
Then $\gamma'+(c_{11}-c_{22})\gamma-c_{21}\gamma^2+c_{12}=0$.
If $u(x)\neq 0$, denote $\delta(x)=v(x)/u(x)$.
Then $\delta'+(c_{22}-c_{11})\delta-c_{12}\delta^2+c_{21}=0$.

So we get the following result.
\begin{pro}Let $C(x)=\begin{pmatrix}
c_{11}(x)&c_{12}(x)\\
c_{21}(x)&c_{22}(x)
\end{pmatrix}\in\C[x]^{2\times 2}$ and $a\neq 0$. Then
$V_{a,b,C(x)}$ is irreducible if and only if $\gamma'+(c_{11}-c_{22})\gamma-c_{21}\gamma^2+c_{12}=0$ and $\delta'+(c_{22}-c_{11})\delta-c_{12}\delta^2+c_{21}=0$ do not have solutions in $\C(x)$.
\end{pro}
}
\delete{
\begin{thm}Let $a\in \C\setminus\{0\}$, $b\in \C$ and $C(x)=\begin{pmatrix}
c_{11}(x)&c_{12}(x)\\
c_{21}(x)&c_{22}(x)
\end{pmatrix}\in\C[x]^{2\times 2}$. Then
$V_{a,b,C(x)}$ is irreducible if and only if $c_{21}(x)\neq0$ and $\gamma'(x)+(c_{11}(x)-c_{22}(x))\gamma(x)-c_{21}(x)\gamma(x)^2+c_{12}(x)=0$ does not have solutions in $\C(x)$.
\end{thm}
\begin{proof}
Note that a proper $\mathcal{L}$-invariant subspace of $K^2$ must be of dimension one over $K$. Let $(u(x),v(x))^T\neq 0$. Consider the equation $$\mathcal{L}\begin{pmatrix}u(x)\\v(x)\end{pmatrix}=\begin{pmatrix}u'(x)\\v'(x)\end{pmatrix}
+C(x)\begin{pmatrix}u(x)\\v(x)\end{pmatrix}=
\alpha(x)\begin{pmatrix}u(x)\\v(x)\end{pmatrix}$$ with $\alpha(x)\in K$.
Then we have
$$\begin{cases}
u'(x)+c_{11}(x)u(x)+c_{12}(x)v(x)=\alpha(x)u(x)\\
v'(x)+c_{21}(x)u(x)+c_{22}(x)v(x)=\alpha(x)v(x)
\end{cases}.
$$
If $v(x)\neq 0$, denote $\gamma(x)=u(x)/v(x)$.
Then we have $\gamma'(x)+(c_{11}(x)-c_{22}(x))\gamma(x)-c_{21}(x)\gamma(x)^2+c_{12}(x)=0$.
If $u(x)\neq 0$, denote $\delta(x)=v(x)/u(x)$.
Then $\delta'(x)+(c_{22}(x)-c_{11}(x))\delta(x)-c_{12}(x)\delta^2(x)+c_{21}(x)=0$.
Then
$V_{a,b,C(x)}$ is irreducible if and only if $\gamma'(x)+(c_{11}(x)-c_{22}(x))\gamma(x)-c_{21}(x)\gamma^2(x)+c_{12}(x)=0$ and $\delta'(x)+(c_{22}(x)-c_{11}(x))\delta(x)-c_{12}(x)\delta(x)^2+c_{21}(x)=0$ do not have solutions in $\C(x)$.

If $c_{21}(x)=0$, then $\delta'(x)+(c_{22}(x)-c_{11}(x))\delta(x)-c_{12}(x)\delta(x)^2=0$ always has a solution $\delta(x)=0$. If $c_{21}(x)\neq0$ and $\delta'(x)+(c_{22}(x)-c_{11}(x))\delta(x)-c_{12}(x)\delta(x)^2+c_{21}(x)=0$, then $\delta(x)\neq 0$.
Denote $\gamma(x)=1/\delta(x)$. Obviously, $\gamma'(x)+(c_{11}(x)-c_{22}(x))\gamma(x)-c_{21}(x)\gamma^2(x)+c_{12}(x)=0$.
\end{proof}
\yy{I rewrite this theorem and this proof as follows.}}
\begin{thm}\label{thm-rk2}Let $a\in \C\setminus\{0\}$, $b\in \C$ and $C(x)=\begin{pmatrix}
c_{11}(x)&c_{12}(x)\\
c_{21}(x)&c_{22}(x)
\end{pmatrix}\in\C[x]^{2\times 2}$. Then
$V_{a,b,C(x)}$ is irreducible if and only if $c_{21}(x)\neq0$, $c_{12}(x)\neq0$ and $\gamma'(x)+(c_{11}(x)-c_{22}(x))\gamma(x)-c_{21}(x)\gamma^2(x)+c_{12}(x)=0$ does not have solutions in $\C(x)$.
\end{thm}
\begin{proof}
By Theorem \ref{thm-c}, for $a\neq 0$, $V_{a,b,C(x)}$ is irreducible if and only if $\mathcal L=\frac{d}{dx}+C(x)$ acts irreducibly on $K^2$, where $K=\mathbb{C}(x)$.

A proper $\mathcal L$-invariant subspace of $K^2$, if it exists, must be one-dimensional. Let it be spanned by $(u(x),v(x))^T\neq 0$. Then there exists $\alpha(x)\in K$ such that $$\mathcal{L}\begin{pmatrix}u(x)\\v(x)\end{pmatrix}=\begin{pmatrix}u'(x)\\v'(x)\end{pmatrix}
+C(x)\begin{pmatrix}u(x)\\v(x)\end{pmatrix}=
\alpha(x)\begin{pmatrix}u(x)\\v(x)\end{pmatrix},$$
which is equivalent to the system
\begin{eqnarray}
\begin{cases}\label{cs}
u'(x)+c_{11}(x)u(x)+c_{12}(x)v(x)=\alpha(x)u(x)\\
v'(x)+c_{21}(x)u(x)+c_{22}(x)v(x)=\alpha(x)v(x)
\end{cases}.
\end{eqnarray}

We first observe that Eq. $(\ref{cs})$ has a nonzero solution in $K$ for some $\alpha(x)\in K$ if and only if at least one of the following two equations has a solution in $K$:
\begin{eqnarray}
&&\label{eqq1}\gamma'(x)+(c_{11}(x)-c_{22}(x))\gamma(x)-c_{21}(x)\gamma^2(x)+c_{12}(x)=0, \\
&&\label{eqq2}\delta'(x)+(c_{22}(x)-c_{11}(x))\delta(x)-c_{12}(x)\delta^2(x)+c_{21}(x)=0.
\end{eqnarray}
Indeed, if $v(x)\neq 0$, setting $\gamma(x)=u(x)/v(x)$ yields Eq. (\ref{eqq1}); if $u(x)\neq 0$, setting $\delta(x)=v(x)/u(x)$ yields a solution of Eq. (\ref{eqq2}). Conversely, if $\gamma$ solves Eq. (\ref{eqq1}), then $(u(x),v(x))^T=(\gamma(x),1)^T$ (with a suitable $\alpha(x)$) is a nonzero solution of Eq. $(\ref{cs})$ with $\alpha(x)=c_{21}(x)\gamma(x)+c_{22}(x)$; similarly, if $\delta(x)$ solves Eq. (\ref{eqq2}), then $(u(x),v(x))^T=(1,\delta(x))^T$ is a nonzero solution of Eq. $(\ref{cs})$ with $\alpha(x)=c_{12}(x)\delta(x)+c_{11}(x)$.

If $c_{21}(x)=0$, then Eq. (\ref{eqq2}) has the solution $\delta(x)=0$. Hence Eq. $(\ref{cs})$ has a nonzero solution. Therefore, the action of $\mathcal L$ on $K^2$ is reducible. By Theorem \ref{thm-c}, $V_{a,b,C(x)}$ is reducible. Similarly, we also obtain that if $c_{12}(x)=0$, then $V_{a,b,C(x)}$ is reducible.

Assume that $c_{21}(x)\neq 0$ and $c_{12}(x)\neq 0$. We claim that in this case, Eq. (\ref{eqq1}) has a solution in $K$ if and only if Eq. (\ref{eqq2}) has a solution in $K$. Suppose that $\gamma(x)$ is a solution of Eq. (\ref{eqq1}). If $\gamma(x)\neq 0$, then $\delta(x)=1/\gamma(x)$ is a solution of Eq. (\ref{eqq2}); if $\gamma(x)=0$, then Eq. (\ref{eqq1}) forces $c_{12}(x)=0$, a contradiction. Conversely, suppose that $\delta(x)$ is a solution of Eq. (\ref{eqq2}). Then $\delta(x)\neq 0$, for otherwise Eq. (\ref{eqq2}) would imply $c_{21}(x)=0$, contradicting our assumption. Then $\gamma(x)=1/\delta(x)$ is a nonzero solution of Eq. (\ref{eqq1}). This proves the claim.

Consequently, when $c_{21}(x)\neq 0$ and $c_{12}(x)\neq 0$, Eqs. (\ref{eqq1}) and (\ref{eqq2}) have solutions simultaneously. Then this conclusion holds.
\end{proof}
\begin{rmk}
Note that Eq. (\ref{eqq1}) has a solution in $K$ for the matrix $C(x)=\begin{pmatrix}
c_{11}(x)&c_{12}(x)\\
c_{21}(x)&c_{22}(x)
\end{pmatrix}$ if and only if Eq. (\ref{eqq2}) has a solution in $K$ for the matrix $C^\star(x)=\begin{pmatrix}
c_{22}(x)&c_{21}(x)\\
c_{12}(x)&c_{11}(x)
\end{pmatrix}$. Therefore, by the proof of Theorem \ref{thm-rk2}, we have that $V_{a,b,C(x)}$ is irreducible if and only if $V_{a,b,C^\star(x)}$ is irreducible.
\end{rmk}
The equation
$\gamma'(x)+(c_{11}(x)-c_{22}(x))\gamma(x)-c_{21}(x)\gamma^2(x)+c_{12}(x)=0$ is a Raccati equation, which can be studied through the following procedure in general. If $\gamma_1(x)$ is a particular solution, then $y(x)=\gamma(x)-\gamma_1(x)$ is a solution of a certain Bernoulli equation, that can be solved explicitly. Thus the key point is to find a particular solution. In order to find a particular solution of Eq. (\ref{eqq1}), by using $\gamma(x)=\frac{-z'(x)}{c_{21}(x)z(x)}$, we obtain a second order linear differential equation $z''(x)+(c_{11}(x)-c_{22}(x))z'(x)-c_{12}(x)c_{21}(x)z(x)=0$. Setting $z(x)=\beta(x) w(x)$ with $2\beta'(x)+(c_{11}(x)-c_{22}(x))=0$, then we get $w''(x)+r(x)w(x)=0$ for some $r(x)\in\C(x)$.  which can be solved by employing Kovacic's algorithm in differential Galois theory \cite{CH}. Alternatively, one can also analyze the possible rational solutions of certain Riccati equations directly.

\begin{ex}
Take $C(x)=\begin{pmatrix}1&x\\
1&1
\end{pmatrix}$. In this case, $c_{21}(x)=1\neq0$ and $c_{12}(x)=x\neq0$. If $\gamma(x)=g(x)/f(x)\in\C(x)$ with $g(x)$, $f(x)\in \C[x]$ is a solution of the ordinary differential equation $$\gamma'(x)+(c_{11}(x)-c_{22}(x))\gamma(x)-c_{21}(x)\gamma^2(x)+c_{12}(x)
=\gamma'(x)-\gamma^2(x)+x=0,$$ then $g'(x)f(x)-f'(x)g(x)-g^2(x)+xf^2(x)=0$.
By comparing the degrees on both sides, we obtain a contradiction. Therefore, the Riccati equation has no solutions in $\mathbb C(x)$. By Theorem \ref{thm-rk2}, $V_{a,b,C(\partial)}$ is irreducible for any $a\neq0$.
\end{ex}

Finally, we shall prove that there exist finite irreducible generalized conformal modules of arbitrary finite rank over $Vir$.

Let $C(x) =\begin{pmatrix} 0 & x  \\  I_{n-1} & 0 \end{pmatrix}\in\C[x]^{n\times n}$. We consider the irreducibility of $V_{a,b,C(x)}$ when $a\neq0$. As explained in Remark \ref{rmk-4.2}, we can take $(a,b)=(1,0)$.
Recall $V_{1,0,C(x)}=\C[x]^n$ and let $e_i=(0,\cdots,0,1,0,\cdots,0)^T, 1\leq i\leq n$ be
the $i$-th coordinate vector of $\C[x]^n$. Then we have
$$C(x)e_i=
\begin{cases}
e_{i+1},&i\leq n-1, \\
xe_1,&i=n.
\end{cases}$$
Hence for any $f(x)\in\C[x]$, we obtain
$$\mathcal{L}(f(x)e_i)=
\begin{cases}
f'(x)e_i+f(x)e_{i+1},&i\leq n-1, \\
f'(x)e_n+xf(x)e_1, &i=n.
\end{cases}$$
Since $L_0=x\mathcal{L}+1$ when acting on $\C[x]^n$, we deduce
\begin{equation}\label{rmk-eq1}
L_0(f(x)e_i)=
\begin{cases}
(f(x)+xf'(x))e_i+xf(x)e_{i+1},&i\leq n-1, \\
(f(x)+xf'(x))e_n+x^2f(x)e_1, &i=n.
\end{cases}
\end{equation}

Throughout the rest part of this section, we assume that $U$ is a nonzero submodule of $V_{1,0,C(x)}$.
\begin{lem}
There exists some nonzero $\sum_{i=1}^n f_i(x)e_i\in U$ such that $deg\,f_n(x)\geq deg\,f_i(x)$ for any $i\in \{1, \ldots,n\}$.
\end{lem}
\begin{proof}
Since $U\neq 0$, there exists some nonzero $u=\sum_{i=1}^n f_i(x)e_i\in U$, where $f_i(x)\in \C[x]$. Denote $d(u)=max\{deg\,f_i(x)\,|\, i=1,\ldots,n\}$ and $\chi(u)=max\{j\,|\,deg\,f_j(x)=d(u)\}=p$. If $p<n$, by Eq.~(\ref{rmk-eq1}),
\delete{\begin{equation*}
\begin{aligned}
L_0(f e_i) &= (f + x f') e_i + x f e_{i+1}, \quad 1 \le i \le n-1, \\
L_0(f e_n) &= (f + x f') e_n + x^2 f e_1,
\end{aligned}
\end{equation*}}
we conclude that the $e_{p+1}$-component of $L_0(u)$ is $xf_p(x)+f_{p+1}(x)+xf'_{p+1}(x)$, which has the highest degree among all the coefficients of $L_0(u)$.  Then $\chi(L_0(u))=p+1$. Replacing $u$ by $L_0(u)$ and continuing this procedure, we get the desired assertion.
\end{proof}

\begin{lem}\label{lem-l2}
There exists some nonzero $h(x)\in \C[x]$ such that $h(x)e_i\in U$ for any $i\in \{1, \ldots,n\}$.
\end{lem}
\begin{proof}
Assume that $u=\sum_{i=1}^n f_i(x)e_i\in U$ with $deg\,f_n(x)=max\{deg\,f_i\,|\, i=1,\ldots,n\}=N$. For convenience, we introduce a linear map $\phi$ defined by $$\phi:V_{1,0,C(x)}\rightarrow V_{1, 0, C(x)}, u\mapsto L_0(u)-u.$$ Then by Eq.~(\ref{rmk-eq1}), we obtain\begin{equation*}
\begin{aligned}
\phi(f(x) e_i)=
\begin{cases}
x f'(x)e_i + x f(x) e_{i+1}, &i=1,\cdots,n-1, \\
x f'(x)e_n + x^2 f(x) e_1,i=n.
\end{cases}
\end{aligned}
\end{equation*} and $\phi(U)\subseteq U$. By direct computation, we get
$$\phi(u)=(x^2f_n(x)+xf'_1(x))e_1+\sum_{i=2}^n (x f'_i(x)+xf_{i-1}(x)) e_i\in U.$$
Therefore, the $e_1$-coefficient of $\phi(u)$ has the
degree $N+2$, strictly higher than any other $e_i$-coefficient of $\phi(u)$ for $i=2,\cdots,n$.
It is easy to check that the leading term of $e_k$-coefficient of $\phi^k(u)$ is $x^{k+1}f_n(x)$, which has
the degree $N+k+1$.
The degree of $e_k$-coefficient of $\phi^k(u)$ is highest among those of $e_i$-coefficient of $\phi^k(u)$, $i=1,\ldots,k$ and
strictly higher than the degree of any other $e_i$-coefficient of $\phi^k(u)$, $i=k+1,\ldots,n$. 

Since $e_1,e_2,\ldots,e_n$ is a $\C[x]$-basis of $V$, we have
$(u,\phi(u),\cdots,\phi^{n-1}(u))=(e_1,e_2,\ldots,e_n)G$, where $G=(g_{ij}(x))\in\C[x]^{n\times n}$.
Note that
$G$ has the form
\begin{equation}
G=
\begin{pmatrix}
\ast &x^2f_n(x)+\text{lower terms}&\ast&\cdots&\ast\\
\ast &\ast&x^3f_n(x)+\text{lower terms}&\cdots&\ast\\
\vdots&\vdots&\vdots&\ddots&\vdots\\
\ast&\ast&\ast&\cdots&x^nf_n(x)+\text{lower terms}\\
f_n(x)&\ast&\ast&\cdots&\ast
\end{pmatrix}.
\end{equation}
Note that the determinant of $G$ has degree $nN+\frac{n(n+1)}{2}$. Therefore, we have $\text{det}~G\neq 0$. Set $h(x)=\text{det}~G$.
Let $G^*$ be the adjoint matrix of $G$. Then we get $GG^*=h(x)I_n$. Since
$$(u,\phi(u),\ldots,\phi^{n-1}(u))G^*=(e_1,e_2,\ldots,e_n)GG^*=(h(x)e_1,h(x)e_2,\ldots,h(x)e_n),$$
we conclude that $h(x)e_i\in U$ for $i=1,\ldots,n$.
\end{proof}

\begin{lem}\label{lem-l3} As a $\C[x]$-module, the rank of $U$ is $n$.
\end{lem}
\begin{proof}
It follows from Lemma \ref{lem-l2} and the fact $\C[x]$ is a principal ideal domain immediately.
\end{proof}
\begin{pro}\label{thm-4} Let $C(\partial) =\begin{pmatrix} 0 & \partial  \\  I_{n-1} & 0 \end{pmatrix}\in\C[\partial]^{n\times n}$ and $a\neq 0$. Then the generalized conformal module $V_{a,b,C(\partial)}$ over $Vir$ is irreducible.
\end{pro}
\begin{proof}
It follows from Remark \ref{rmk-4.2}, Lemma \ref{lem-l3} and Theorem \ref{thm-c}  immediately.
\end{proof}

\begin{thm}\label{thm-5}
There exist irreducible generalized conformal modules of arbitrary rank over $Vir$.
\end{thm}
\begin{proof}
It is straightforward by Proposition \ref{thm-4}.
\end{proof}
\section{A class of infinite generalized conformal modules over $Vir$ which are torsion as $\C[\partial]$-modules}
In this section, we construct a family of irreducible generalized conformal modules over $Vir$ that are torsion over $\C[\partial]$ yet non-trivial as $Vir$-modules. Such modules do not exist in the finite case (see Proposition \ref{pro-finite}), and they also have no counterpart among conformal modules. Thus, they reveal how the theory of generalized conformal modules differs from that of conformal modules, and how the infinite behavior diverges from the finite one.

Recall that a finite irreducible non-trivial conformal module over $Vir$ is of rank one.
A rank one module over $Vir$ has a $\C$-basis $v_i, i=0,1,2,\ldots, n, \ldots$ with the action $\partial v_i=v_{i+1}$. The $\C[\partial]$-module structure can be indicated as
$$v_0\rightarrow v_1\rightarrow \cdots\rightarrow v_i\rightarrow v_{i+1}\rightarrow \cdots.$$
Reversing the arrow of the above diagram, we have
$$v_0\leftarrow v_1\leftarrow \cdots\leftarrow v_i\leftarrow v_{i+1}\leftarrow \cdots.$$
This represents a new $\C[\partial]$-module, which is called a {\bf $\C[\partial]$-module of corank one}.
Explicitly, a $\C[\partial]$-module $V$ of corank one is a vector space with a $\C$-basis $v_i, i=0,1,2,\ldots$ and the action
$\partial v_i=v_{i-1}$ (denote $v_{-1}=0$ for convenience).
Note that $V$ is not a finitely generated $\C[\partial]$-module and $\text{Tor}(V)=V$.

We are interested in the following question: {\bf can a $\C[\partial]$-module of corank one be made into a non-trivial module over $Vir$?}
In order to tackle this problem, we realize $V$ as $\mathbb C[t]$ by sending $v_i$ to $t^i/i!$, with $\partial$ acting as $d/dt$.
We begin by studying the conformal maps of $\C[t]$.

\begin{lem}\label{lem-conf}
Let $\Phi_\lambda: \C[t]\rightarrow \C[t][[\lambda]]$ be a conformal map, i.e., a linear
map satisfying $\Phi_\lambda(\partial f(t))=(\partial+\lambda)\Phi_\lambda(f(t))$ for any $f(t)\in\C[t]$.
Then there exist $b_i(\lambda)\in\C[[\lambda]],i\in\Z_+$ such that $$\Phi_\lambda(f(t))=e^{-\lambda t}(\sum_{i=0}^\infty b_i(\lambda)\partial^i)(f(t))\;\;\;\text{for any $f(t)\in\C[t]$}.$$
\end{lem}
\begin{proof}
Since $\partial$ acts locally nilpotently on $\C[t]$, $\sum_{i=0}^\infty b_i(\lambda)\partial^i$ is well defined on $\C[t]$.
For $n\in \Z_+$, denote $\Phi_\lambda(t^n)=\chi_n(t,\lambda)$, where $\chi_n(t,\lambda)\in\C[t][[\lambda]]$.
Then we have $\Phi_\lambda(\partial t^n)=n\Phi_\lambda(t^{n-1})=n\chi_{n-1}(t,\lambda)$ and $(\partial+\lambda)\Phi_\lambda(t^n)=\chi_n'(t,\lambda)+\lambda \chi_n(t,\lambda)$. Consequently, we obtain
$$\chi_n'(t,\lambda)+\lambda \chi_n(t,\lambda)=n\chi_{n-1}(t,\lambda).$$
Denote $\chi_n(t,\lambda)=e^{-\lambda t}\rho_n(t,\lambda)$. Then we have
$\rho_n'(t,\lambda)=n\rho_{n-1}(t,\lambda)$. Set $a_n(\lambda)=\rho_n(0,\lambda)$.
\begin{cla}
$\rho_n(t,\lambda)=\sum_{i=0}^n \binom{n}{i}a_i(\lambda)t^{n-i}$.
\end{cla}
We prove this claim by induction on $n$.
Since $\rho_0'(t,\lambda)=0$, we have $\rho_0(t,\lambda)=\rho_0(0,\lambda)=a_0(\lambda)$. Assume
$\rho_{n-1}(t,\lambda)=\sum_{i=0}^{n-1} \binom{n-1}{i}a_i(\lambda)t^{n-1-i}$.
Since $\rho_n'(t,\lambda)=n\rho_{n-1}(t,\lambda)$, we obtain
\begin{equation*}
\aligned
\rho_n(t,\lambda)&=\int_{0}^{t}n\rho_{n-1}(t,\lambda) dt+\rho_n(0,\lambda)\\
&=\int_{0}^{t}\sum_{i=0}^{n-1}\binom{n-1}{i}na_i(\lambda)t^{n-1-i}dt+a_n(\lambda)\\
&=\sum_{i=0}^n \binom{n}{i}a_i(\lambda)t^{n-i}.
\endaligned
\end{equation*}
Then the claim follows.

Note that $\sum_{i=0}^n \binom{n}{i}a_i(\lambda)t^{n-i}=(\sum_{i=0}^\infty a_i(\lambda)\frac{\partial^i}{i!})t^n$. Hence, we get
$$\Phi_\lambda(t^n)=\chi_n(t,\lambda)=e^{-\lambda t}\rho_n(t,\lambda)=(e^{-\lambda t}\sum_{i=0}^\infty a_i(\lambda)\frac{\partial^i}{i!})t^n.$$
Set $b_i(\lambda)=\frac{a_i(\lambda)}{i!}$.
Then this conclusion follows.
\end{proof}

Let $\C[t]$ be a generalized conformal module over $Vir$.
Since $L_\lambda: \C[t]\rightarrow \C[t][[\lambda]]$ is a linear
map satisfying $\Phi_\lambda(\partial t^n)=(\partial+\lambda)\Phi_\lambda(t^n)$ for any $n\in \Z_+$, by Lemma \ref{lem-conf},
there
exist
$b_i(\lambda)\in\C[[\lambda]],i\in\Z_+$ such that $L_\lambda f(t)=e^{-\lambda t}(\sum b_i(\lambda)\partial^i) f(t)$.
Then we have
\begin{equation*}
[L_\lambda L]_{\lambda+\mu}f(t)=(\lambda-\mu)L_{\lambda+\mu}f(t)=(\lambda-\mu)e^{-\lambda t-\mu t}(\sum_i b_i(\lambda+\mu)\partial^i) f(t).
\end{equation*}
Moreover, we obtain
\begin{equation*}
\aligned
L_\lambda L_{\mu}f(t)=&L_\lambda(e^{-\mu t}(\sum b_i(\mu)\partial^i) f(t))\\
=&e^{-\lambda t}(\sum b_j(\lambda)\partial^j))(e^{-\mu t}(\sum b_i(\mu)\partial^i) f(t))\\
=&e^{-\lambda t-\mu t}\sum_i\sum_j\sum_k \binom{j}{k}b_j(\lambda)b_i(\mu)(-\mu)^k\partial^{i+j-k} f(t)\\
=&e^{-\lambda t-\mu t}\sum_r\sum_s\sum_p \binom{r+s}{s}b_{r+s}(\lambda)b_{p-r}(\mu)(-\mu)^s\partial^{p} f(t).
\endaligned
\end{equation*}
Similarly, we have
\begin{equation*}
L_{\mu}L_\lambda f(t)=e^{-\lambda t-\mu t}\sum_r\sum_s\sum_p \binom{r+s}{s}b_{r+s}(\mu)b_{p-r}(\lambda)(-\lambda)^s\partial^{p} f(t).
\end{equation*}
Hence by $[L_\lambda L]_{\lambda+\mu}t^n=L_\lambda L_{\mu}t^n-L_{\mu}L_\lambda t^n$, we have
\begin{equation}\label{tor1}
(\lambda-\mu)b_p(\lambda+\mu)=
\sum_{j=0}^n\sum_{k=0}^j\binom{j}{k}b_{j}(\lambda)b_{p-j+k}(\mu)(-\mu)^k-\sum_{j=0}^n\sum_{k=0}^j\binom{j}{k}b_{j}(\mu)b_{p-j+k}(\lambda)(-\lambda)^k.
\end{equation}
Set $F(z, \lambda)=\sum_{p=0}^\infty b_p(\lambda)z^p$. Then Eq.~(\ref{tor1}) becomes
\begin{equation}\label{tor2}
(\lambda-\mu)F(z,\lambda+\mu)=F(z-\mu,\lambda)F(z,\mu)-F(z-\lambda,\mu)F(z,\lambda).
\end{equation}
Set $G(z,\lambda)=F(z,-\lambda)$. Then Eq.~(\ref{tor2}) becomes
\begin{equation}\label{tor3}
(\lambda-\mu)G(z,\lambda+\mu)=G(z+\lambda,\mu)G(z,\lambda)-G(z+\mu,\lambda)G(z,\mu).
\end{equation}
Next, we solve this equation.
 We remark that Eq. (\ref{tor3}) is the same as
Eq. (\ref{main}), unless $G(z,\lambda)$ is required to be in $\C[[z,\lambda]]$, not in $\C[z][[\lambda]]$.
\begin{rmk}
For any $q(z)\in \C[[z]]$, 
$e^{q(z+\lambda)-q(z)}$ is also
well defined and $e^{q(z+\lambda)-q(z)}\in\C[[z,\lambda]]$.
\end{rmk}
Using the same method as in the proof of Lemma \ref{thm-1}, we have
$$G(z, \lambda)=(z+a \lambda+b)e^{q(z+\lambda)-q(z)}$$ for some $a,b\in\C$ and some $q(z)\in\C[[z]]$.
Then we obtain
\begin{eqnarray*}
L_\lambda f(t)&=&e^{-\lambda t}(\sum_{i=0}^\infty b_i(\lambda)\partial^i) f(t)=e^{-\lambda t}F(\partial,\lambda)f(t)\\
&=&e^{-\lambda t}G(\partial,-\lambda)f(t)=e^{-\lambda t}(\partial-a \lambda+b)e^{q(\partial-\lambda)-q(\partial)}f(t).
\end{eqnarray*}
Denote this module by $U_{a,b,q}$. 
When $q=0$, set $U_{a,b}=U_{a,b,0}$.
\begin{lem}For any $q(z)\in\C[[z]]$,
$U_{a,b}\cong U_{a,b,q}$.
\end{lem}
\begin{proof}
\delete{Recall that
$$U_{a,b,q}=\C[t],\ L_\lambda f(t)=e^{-\lambda t}(\partial-a \lambda+b)e^{q(\partial-\lambda)-q(\partial)}f(t)$$
and
$$U_{a,b}=\C[t],\ L_\lambda f(t)=e^{-\lambda t}(\partial-a \lambda+b)f(t).$$}
Define $$\Phi:U_{a,b}\rightarrow U_{a,b,q},\ f(t)\mapsto e^{q(\partial)}f(t)\;\; \text{for any $f(t)\in U_{a,b}$.}$$ Then $\Phi$ is a well defined bijective $\C[\partial]$-module homomorphism.
By  direct computation, we have
\[
\Phi(L_\lambda f(t))=e^{q(\partial)}(L_\lambda f(t))=L_\lambda (e^{q(\partial-\lambda)}f(t))=e^{-\lambda t}(\partial-a \lambda+b)e^{q(\partial-\lambda)}f(t),
\]
and
\[
L_\lambda (\Phi f(t))=(L_\lambda e^{q(\partial)}f(t))=e^{-\lambda t}(\partial-a \lambda+b)e^{q(\partial-\lambda)-q(\partial)}e^{q(\partial)}f(t)=e^{-\lambda t}(\partial-a \lambda+b)e^{q(\partial-\lambda)}f(t).
\]
Therefore, $\Phi$ is a module isomorphism from $U_{a,b}$ to $U_{a,b,q}$.
\end{proof}
Therefore, we obtain the following result.
\begin{thm}\label{thm-6}
Let $V$ be a generalized conformal module of corank one over the Virasoro conformal algebra Vir. Then $V\cong U_{a,b}$ for some $a,b\in\C$. Moreover, $U_{a,b}$ is irreducible if and only $(a,b)\neq (0,0)$.
\end{thm}
\begin{proof}
We only need to determine 
the irreducibility of $U_{a,b}$.
If $(a,b)=(0,0)$, then we have $$L_\lambda f(t)=e^{-\lambda t}\partial f(t)=e^{-\lambda t}f'(t).$$
Consequently, $L_\lambda 1=0$ and $\C 1$ is a proper submodule of $U_{0,0}$. Hence $U_{0,0}$ is reducible.
If $(a,b)\neq(0,0)$,  then we obtain $$L_\lambda f(t)=e^{-\lambda t}(\partial-a \lambda+b)f(t).$$
Assume that $U$ is a nonzero submodule of $U_{a,b}$.
Since $U$ is a $\C[\partial]$-module, we have $1\in U$.
Moreover, we obtain
$$L_\lambda 1=e^{-\lambda t}(-a \lambda+b)
=(a \lambda+b)\sum_{k=0}^\infty \frac{(-\lambda t)^k}{k!}
=\sum_{k=0}^\infty(ka(-t)^{k-1}+b(-t)^k)\frac{\lambda^k}{k!}\in U[[\lambda]].$$
We deduce that $t^k\in U$ for any $k\in\Z_+$. Then $U=U_{a,b}$ and $U_{a,b}$ is irreducible.
\end{proof}
\begin{pro}Let $a_1,a_2,b_1,b_2\in\C$. Then
$U_{a_1,b_1}\cong U_{a_2,b_2}$ if and only if $b_1=b_2$ and one of the following holds:
\begin{enumerate}
\item $a_1=a_2$,
\item $b_1\neq0$ and $(a_1,a_2)=(1,0)$,
\item $b_1\neq0$ and $(a_1,a_2)=(0,1)$.
\end{enumerate}
\end{pro}
\begin{proof}
Assume that $U_{a_1,b_1}\cong U_{a_2,b_2}$ and $\Phi$ is such an isomorphism.
Since $\Phi$ is a $\C[\partial]$-module homomorphism of $\C[t]$, by a similar argument as in the proof of Lemma \ref{lem-conf},
we have $\phi=P(\partial)=\sum_{i=0}^\infty p_k \partial^k\in\C[[\partial]]$ with $p_k\in\C$.
Moreover, since $\Phi$ is bijective, we have $p_0\neq 0$.
By $\Phi(L_\lambda f(t))=L_\lambda(\Phi(f(t))$, we obtain
\begin{equation}\label{eq-cr1}
P(\partial-\lambda)(\partial-a_1 \lambda+b_1)=P(\partial)(\partial-a_2 \lambda+b_2).
\end{equation}
Setting $\lambda=0$, we have $b_1=b_2$. Write $b=b_1$. Setting $\partial=0$ in Eq. (\ref{eq-cr1}), we have
$P(-\lambda)(-a_1 \lambda+b)=p_0(-a_2 \lambda+b)$. Replacing $\lambda$ by $-\lambda$ gives
$P(\lambda)(a_1 \lambda+b)=p_0(a_2 \lambda+b)$. The discussion varies on whether $b=0$.

If $b=0$, we have $a_1P(\lambda)=a_2p_0$. Then $a_1=a_2$ and $P(\partial)=p_0$.

If $b\neq 0$, since $a_1 \lambda+b$ is invertible in $\C[[\lambda]]$, combining
$$P(\partial)(a_1 \partial+b)=p_0(a_2 \partial+b),\ P(\partial-\lambda)(a_1 \partial-a_1\lambda+b)=p_0(a_2 \partial-a_2\lambda+b)$$ with Eq. (\ref{eq-cr1}), we have
$$(a_1 \partial+b)(a_2 \partial-a_2\lambda+b)(\partial-a_1 \lambda+b)=(a_2 \partial+b)(a_1 \partial-a_1\lambda+b)(\partial-a_2 \lambda+b)$$
which implies $a_1=a_2$, $(a_1, a_2)=(0,1)$ or $(a_1,a_2)=(1,0)$.

Next, we show $U_{0,b}\cong U_{1,b}$ if $b\neq 0$.
Note that when $(a_1,a_2)=(0,1)$ and $b\neq0$, we can take $p_0=b$ and then $P(\partial)=\partial+b$. So $f(t)\mapsto (\partial+b)(f(t))=f'(t)+bf(t)$ gives an isomorphism from $U_{0,b}$ to $U_{1,b}$. And $$P(\partial)^{-1}=(\partial+b)^{-1}=\sum_{k=0}^\infty (-1)^kb^{-k-1}\partial^k\in\C[[\partial]]$$ is an isomorphism from $U_{1,b}$ to $U_{0,b}$.
Then the conclusion holds.
\end{proof}

\begin{rmk}Let $V$ be a generalized conformal module over a Lie conformal algebra $\mathcal{A}$ and $\alpha\in\C$. Then $$V_{\partial-\alpha}=\{v\in V\,|\,(\partial-\alpha)^nv=0\;\;\text{for some $ n\in\Z_+$}\}$$ is a submodule of $V$ by Lemma \ref{lem-tor1}. By Lemma \ref{lem-tor2}, for any irreducible generalized conformal module $V$ with $\text{Tor}~(V)\neq 0$, 
we get $V=V_{\partial-\alpha}$ for some $\alpha\in\C$.
The module $U_{a,b}$ is an example of this form with $\alpha=0$.
Using a similar method, one can also construct an irreducible generalized conformal module over $Vir$ with $V=V_{\partial-\alpha}$ for a general $\alpha\in\C$.
First, we construct a $\C[\partial]$-module with a $\C$-basis $v_i, i=0,1,2,\ldots, n, \ldots$ and the action $(\partial -\alpha)v_i=v_{i-1}$. Let $\C[\partial]$ act on $\C[t]$ by $$\partial(f(t))=(\frac{d}{dt}+\alpha)f(t)=f'(t)+\alpha f(t).$$ Then $\C[t]$ becomes a $\C[\partial]$-module and $\C[t]=\C[t]_{\partial-\alpha}$. In the following, we attempt to define a module action of $Vir$ on $\C[t]$.
Let $\partial'=\partial-\alpha$.
Note that for a linear map $\Phi_\lambda: \C[t]\rightarrow \C[t][[\lambda]]$,  $\Phi_\lambda(\partial f(t))=(\partial+\lambda)\Phi_\lambda(f(t))$ is equivalent to
$\Phi_\lambda(\partial' f(t))=(\partial'+\lambda)\Phi_\lambda(f(t))$. So by an analogous argument as in the proof of Theorem \ref{thm-6}, we deduce this module is isomorphic to the following module
$$U^\alpha_{a,b}=\C[t],\ L_\lambda f(t)
=e^{-\lambda t}(\partial'-a \lambda+b)f(t)
=e^{-\lambda t}(\partial-a \lambda+b-\alpha)f(t),
$$
for any $f(t)\in\C[t]$, where $a,b\in\C$.
Similarly, $U^\alpha_{a,b}$ is irreducible if and only if $(a,b)\neq(0,0)$.
\end{rmk}

We conclude this paper with some applications of Theorem \ref{thm-6} to the representation theory of the Lie algebra $W_1$ of vector fields on a line and $sl(2,\C)$. Since $W_1$ is isomorphic to the annihilation algebra of $Vir$ and $sl(2,\C)$ is a subalgebra of $W_1$, $U_{a,b}$ becomes a module of $W_1$ by Proposition \ref{pro-3} and consequently a module over $sl(2,\C)$. Therefore, we can obtain a large class of interesting modules over these Lie algebras.
Explicitly, by $$L_\lambda f(t)
=e^{-\lambda t}(\partial-a \lambda+b)f(t)=\sum_{k=0}^\infty((-t)^k(f'(t)+bf(t))+(-t)^{k-1}k(-af))\frac{\lambda^k}{k!},$$
we have $L_{(k)}f(t)=(-1)^k(t^k(f'(t)+bf(t))+kat^{k-1}f(t))$. Then $U_{a,b}=\C[t]$ becomes a module of $W_1$ with actions
$$L_{k}f(t)=(-1)^{k+1}(t^{k+1}(f'(t)+bf(t))+(k+1)at^{k}f(t))\;\;\text{for any $k\geq-1$,}$$ In other words, we have
$$L_{k}t^n=(-1)^{k+1}((n+(k+1)a)t^{k+n}+bt^{k+n+1})\;\;\;\text{for any $k\geq-1$ and $ n\in\Z_+$.}$$ This module is irreducible if and only $(a,b)\neq(0,0)$.
Using the injection $f\mapsto L_{-1}, h\mapsto -2L_0, e\mapsto -L_1$ from $sl(2,\C)$ to $W_1$, we obtain a representation of $sl(2,\C)$ in $\C[t]$ by
\begin{equation*}
\begin{aligned}
&f t^n =L_{-1}t^n=nt^{n-1}+bt^n, \\
&ht^n=-2L_0t^n=2(n+a)t^n+2bt^{n+1},\\
&et^n=-L_1t^n=-(n+2a)t^{n+1}-bt^{n+2}.
\end{aligned}
\end{equation*}
It is easy to check that this module is irreducible if and only if $b\neq0$ or $a\notin-\frac{1}{2}\Z_+$.

\smallskip
\noindent {\bf Acknowledgments.} This research is supported by the Zhejiang
Provincial Natural Science Foundation of China (No.~LZ25A010004) and
Natural Science Foundation of China (No. 12171129).

\smallskip
\noindent
{\bf Declaration of interests. } The authors have no conflicts of interest to disclose.

\smallskip

\noindent
{\bf Data availability. } No new data were created or analyzed in this study.

\end{document}